\documentclass[english,11pt]{article}%
\usepackage[T1]{fontenc}
\usepackage[latin9]{inputenc}
\usepackage{amsmath}
\usepackage{amsthm}
\usepackage{amssymb}
\usepackage{amsfonts}
\usepackage{babel}
\usepackage{fancyhdr}
\usepackage{color}
\usepackage[toc]{appendix}
\usepackage{graphicx}%
\providecommand{\U}[1]{\protect\rule{.1in}{.1in}}
\makeatletter
\newtheorem{theorem}{Theorem}

\newtheorem{condition}[theorem]{Condition}

\newtheorem{definition}[theorem]{Definition}

\newtheorem{lemma}[theorem]{Lemma}

\theoremstyle{remark}
\theoremstyle{definition}
\newtheorem{example}[theorem]{Example}
\newtheorem{remark}[theorem]{Remark}
\numberwithin{equation}{section}
\numberwithin{theorem}{section}

\makeatother
\begin{document}

\title{Analysis and optimization of certain parallel Monte Carlo methods in the low
temperature limit}
\author{Paul Dupuis\thanks{Division of Applied Mathematics, Brown University,
Providence, USA. Research supported in part by the National Science Foundation (DMS-1904992) and the AFOSR (FA-9550-18-1-0214).} \, and Guo-Jhen Wu\thanks{Department of
Mathematics, KTH Royal Institute of Technology, Stockholm, Sweden. Research supported in part by the AFOSR (FA-9550-18-1-0214).} }

\maketitle

\begin{abstract}
Metastability is a formidable challenge to Markov chain Monte Carlo methods. In this paper we present methods for algorithm design to meet this challenge.
The design problem we consider is temperature selection for the infinite swapping scheme, which is 
the limit of the widely used parallel tempering scheme obtained when the swap rate tends to infinity.
We use a recently developed tool for the analysis of the empirical measure of a small noise diffusion to 
transform the variance reduction problem 
into an explicit optimization problem. 
Our first analysis of the optimization problem is
in the setting of a double well model, and it shows that 
the optimal selection of temperature ratios is a geometric sequence except possibly the highest temperature. 
In the same setting we identify two different sources of variance reduction, and show how their competition 
determines the optimal highest temperature. 
In the general multi-well setting we prove that a pure geometric sequence of temperature ratios 
is always nearly optimal, with a performance gap that decays geometrically in the number of temperatures. 
\end{abstract}

\section{Introduction}

\label{sec:introduction}

{Monte Carlo} methods are among the most general purpose stochastic simulation
methods currently available. However, rare events present a particular
challenge for the design of efficient Monte Carlo methods. There is a
relatively long history of the use of large deviation ideas in the design of
algorithms for estimating probabilities of single rare events \cite{deadup,
dupwan5}, since large deviation results can be used to determine how the rare
events are most likely to occur. But less is known on how to adverse overcome
the impact of rare events on Markov chain Monte Carlo (MCMC).

\textit{Parallel tempering} (PT) \cite{swewan, gey}, also known as
\textit{replica exchange}, and a scheme obtained as a suitable limit and known
as \textit{infinite swapping} (INS) \cite{dupliupladol}, are methods for
accelerating MCMC. They work by coupling reversible Markov chains with
different \textquotedblleft temperatures\textquotedblright\ to enhance the
sampling properties of the ensemble. An important question that remains to be
answered is how to choose the temperatures in these algorithms.

In this paper, we apply recently developed methods for the analysis of the
empirical measure of a small noise diffusion to characterize the optimal
temperatures in the low temperature limit, which is the setting where the
difficulties caused by rare events and related metastable behaviors are most
severe. The analysis is done for the INS scheme, which is itself an optimized
limit of parallel tempering, in part because of this optimality, and also in
part because the large deviation properties needed for the analysis take a
simpler form for INS than for PT. However, the conclusions regarding optimal
temperature placements will also be at least approximately valid for parallel
tempering if the swap rate is high enough that it approximates infinite swapping.

In the course of the analysis we are able to identify mechanisms that produce
variance reduction, and find that it has two sources. As will be discussed in
detail later, one source of improved sampling is the increased mobility
obtained by lowering the maximum energy barriers. A second and less obvious
source of variance reduction is due to certain weights appearing in INS, which
play a role reminiscent of the likelihood ratios that appear in importance
sampling (see Section \ref{subsec:source_of_variance_reduction}). As it turns
out, it is the weights that are responsible for most of the variance
reduction, and which ultimately determine the proper placement of the
temperatures in the low temperature limit.

The paper is organized as follows. The problem of interest is described in
Section \ref{sec:problem_formulation}. Various Monte Carlo methods including
PT and INS are discussed in Section \ref{subsec:accelerated_MCMC}, as are the
performance measure we will use to characterize good performance. Section
\ref{sec:statement_of_the_main_results} states the main theoretical results of
the paper, and also includes a discussion on the mechanisms that produce
variance reduction in the accelerated Monte Carlo methods. The proof of
our main result, Theorem \ref{thm:lower_bound_INS},  is given in Section \ref{sec:lower_bound}.
Section \ref{sec:6} gives examples and discusses bounds on crucial parameters that appear in Theorem  \ref{thm:lower_bound_INS}, and
the Appendix sketches the proof of why the INS model satisfies a large
deviation principle on path space.

\section{Problem formulation}

\label{sec:problem_formulation}

We are concerned with computing integrals with respect to a Gibbs measure on
the state space $\mathbb{R}^{d}$. The measure takes the form
\begin{equation}
\mu^{\varepsilon}(dx)\doteq\frac{1}{Z_{\mu}^{\varepsilon}}e^{-\frac{V({x}%
)}{\varepsilon}}dx, \label{eqn:Gibbs}%
\end{equation}
where $V:\mathbb{R}^{d}\rightarrow\mathbb{R}$ is the potential of a complex
physical system, $\varepsilon>0$ is proportional to a parameter that is
interpreted as temperature in physical systems, and the normalization constant
$Z_{\mu}^{\varepsilon}$ is typically unknown.\footnote{To be precise, in a
physical system one would have $\varepsilon=k_{B}T$, where $T$ is the
temperature and $k_{B}$ is Boltzmann's constant, but we abuse terminology and
simplify notation by referring to $\varepsilon$ as a temperature.} As an
elementary example, one would like to estimate $\mu^{\varepsilon}(A)$ for a
set $A\subset\mathbb{R}^{d}$ which does not contain the global minimum of $V$,
with $\partial A$ regular. Problems of this general sort occur in chemistry,
physics, statistics, Bayesian statistics and elsewhere.

Under proper conditions on $V,$ one can check using detailed balance that
$\mu^{\varepsilon}$ is the unique invariant distribution of the diffusion
process $\{X^{\varepsilon}(t)\}_{t\geq0}$ satisfying the stochastic
differential equation%

\begin{equation}
dX^{\varepsilon}\left(  t\right)  =-\nabla V\left(  X^{\varepsilon}\left(
t\right)  \right)  dt+\sqrt{2\varepsilon}dW\left(  t\right)  ,
\label{eqn:dym_original}%
\end{equation}
where $W$ is a $d$-dimensional standard Wiener process.

The empirical measure of $\{X^{\varepsilon}(t)\}_{t\geq0}$ over the time
interval $[0,T]$ is defined by%
\begin{equation}
\lambda^{\varepsilon,T}\left(  dx\right)  \doteq\frac{1}{T}\int_{0}^{T}%
\delta_{X^{\varepsilon}\left(  t\right)  }\left(  dx\right)  dt,
\label{eqn:ordMC}%
\end{equation}
where $\delta_{x}$ is the Dirac measure at $x$. The ergodic theorem implies
$\lambda^{\varepsilon,T}$ gives an approximation to $\mu^{\varepsilon}$, and
strictly speaking it is the use of discrete time analogues in this context
that is known as MCMC, though we will also use the term for the continuous
time model. For the particular problem of approximating $\mu^{\varepsilon}%
(A)$, we have the estimator
\begin{equation}
\theta_{\text{MC}}^{\varepsilon,T}\doteq\lambda^{\varepsilon,T}\left(
A\right)  =\frac{1}{T}\int_{0}^{T}1_{A}\left(  X^{\varepsilon}\left(
t\right)  \right)  dt. \label{eqn:MCMC_estimator}%
\end{equation}
We think of $\theta_{\text{MC}}^{\varepsilon,T}$ as the most straightforward
MCMC estimator of $\mu^{\varepsilon}(A)$, and since we will later on introduce
more complicated estimators, a subscript (e.g., MC) will be used to
distinguish the different estimators.

In many applications (e.g., chemistry, physics, Bayesian inference, counting
\cite{liu, rubkro}), $V(x)$ is a complicated surface which contains multiple
local minima of varying depths. The diffusion $\{X^{\varepsilon}(t)\}_{t\geq
0}$ can be trapped within these deep local minima for a long time before
moving out to other parts of the state space, a phenomena sometimes referred
to a \textit{metastability}. As a result, it requires a very long (exponential
in $1/\varepsilon$) simulation time for $\lambda^{\varepsilon,T}$ to
approximate the equilibrium $\mu^{\varepsilon}$ when $\varepsilon$ is small.

Our analysis of the performance of computational approximations for
$\mu^{\varepsilon}$ will be based on recently derived large deviation
approximations for variances associated with empirical measures such as
(\ref{eqn:MCMC_estimator}) \cite{dupwu}. Following the convention of
\cite[Chapter 6]{frewen2}, \cite{dupwu} considers in place of say
(\ref{eqn:dym_original}) a small noise diffusion that takes values in a
compact and connected manifold $M\subset\mathbb{R}^{d}$ of dimension $r$ and
with smooth boundary (precise regularity assumptions for $M$ are given on
\cite[page 135]{frewen2}). This is also consistent with how MCMC algorithms
for a process such as (\ref{eqn:dym_original}) are often implemented by using
periodic boundary conditions that are far removed for the regions of interest.
However, for ease of discussion we will keep the notation of the SDE model,
but with the understanding that we mean a diffusion process with the same
local characteristics that takes values in the compact space $M$, with $M$
locally equivalent to a Euclidean space.

\begin{remark}
In this paper we focus on the problem of computing integrals with respect to a
Gibbs measure on a continuous state space. However, analogous results for
discrete state systems are expected. See \cite{doldupnyq} for the formulation
of infinite swapping for discrete state models.
\end{remark}

\section{Accelerated MCMC}

\label{subsec:accelerated_MCMC}

In this section we introduce various alternative estimators of $\mu
^{\varepsilon}(A)$ as in (\ref{eqn:Gibbs}). Consider an ergodic Markov process
$\{\bar{X}^{\varepsilon}(t)\}_{t}\subset\bar{M}$ and suppose that
$\nu^{\varepsilon}\in\mathcal{P}(\bar{M})$ is the unique stationary
distribution of $\{\bar{X}^{\varepsilon}(t)\}_{t}.$ As an example, $\bar{M}$
could be $K\in\mathbb{N}$ products of the $M$ just introduced. If we define
$\theta^{\varepsilon,T}$ by
\begin{equation}
\theta^{\varepsilon,T}\doteq\frac{1}{T}\int_{0}^{T}f^{\varepsilon}\left(
\bar{X}^{\varepsilon}\left(  t\right)  \right)  dt \label{eqn:estimator}%
\end{equation}
for a bounded and measurable function $f^{\varepsilon}:\bar{M}%
\rightarrow\mathbb{R}$ such that
\[
\int_{\bar{M}}f^{\varepsilon}\left(  \bar{x}\right)  \nu^{\varepsilon}\left(
d\bar{x}\right)  =\mu^{\varepsilon}(A),
\]
then by the ergodic theorem \cite{bre}, $\theta^{\varepsilon,T}\rightarrow$
$\mu^{\varepsilon}(A)$ w.p.1 as $T\rightarrow\infty$, which means one can also
consider $\theta^{\varepsilon,T}$ as an approximation to $\mu^{\varepsilon
}(A).$ We will consider several classes of estimators that are of the general
form (\ref{eqn:estimator}).

\subsection{Parallel tempering}

\label{subsubsec:parallel_tempering}

{Parallel tempering} is an algorithm used to speed up the sampling of a
\textquotedblleft slowly converging\textquotedblright\ Markov process, i.e.,
one for which the empirical measure converges slowly to the stationary
distribution. Specifically, the idea of two-temperature parallel tempering is
to introduce a \textit{higher temperature}  $\varepsilon/\alpha$ in addition to $\varepsilon$ with $\alpha\in(0,1)$. If $W_{1}$\ and $W_{2}$
are\ independent Wiener processes, then the empirical measure of the pair%
\begin{equation}
\left\{
\begin{array}
[c]{l}%
dX_{1}^{\varepsilon}=-\nabla V(X_{1})dt+\sqrt{2\varepsilon}dW_{1}\\
dX_{2}^{\varepsilon}=-\nabla V(X_{2})dt+\sqrt{2\varepsilon/\alpha}dW_{2}%
\end{array}
,\right.  \label{eqn:dym_PT}%
\end{equation}
gives an approximation to the Gibbs measure with density $\psi^{\varepsilon
}(x_{1},x_{2})\varpropto e^{-V(x_{1})/\varepsilon}e^{-\alpha V(x_{2}%
)/\varepsilon}.$ If we allow \textit{swaps }between $X_{1}^{\varepsilon}$
and $X_{2}^{\varepsilon}$\textit{\ }, i.e., $X_{1}^{\varepsilon}$\ and
$X_{2}^{\varepsilon}$ \textit{exchange locations} with the state dependent
intensity $a\left(  1\wedge\lbrack\psi^{\varepsilon}(x_{2},x_{1}%
)/\psi^{\varepsilon}(x_{1},x_{2})]\right)  $, then we have a \textit{Markov
jump-diffusion}. Moreover, it is straightforward to check this new process
still satisfies detailed balance with respect to $\psi^{\varepsilon}%
(x_{1},x_{2})$ if this swapping intensity is used, and so can be used for
numerical approximations.

It has been shown that various rates of convergence, such as the large
deviation empirical measure rate \cite{dupliupladol} and the asymptotic
variance, can be optimized by letting $a\rightarrow\infty$. This suggests one
should consider the limit as $a\rightarrow\infty$ (the infinite swapping
limit). This cannot be done directly with the parallel tempering processes,
since they will not be tight, and hence do not converge in a meaningful way.
An alternative perspective is to consider a temperature swapping process and
approximate $\psi^{\varepsilon}(x_{1},x_{2})dx_{1}dx_{2}$ by a corresponding
weighted empirical measure instead (see \cite{dupliupladol} for details). The
advantage of doing so is that we have a well defined weak limit process as
$a\rightarrow\infty$, though as noted the empirical measure is replaced by a
weighted analogue. The limit model is as follows. We define $(Y_{1}%
^{\varepsilon},Y_{2}^{\varepsilon})$ as the solution to%
\[
\left\{
\begin{array}
[c]{l}%
dY_{1}^{\varepsilon}=-\nabla V(Y_{1}^{\varepsilon})dt+\sqrt{2\varepsilon\rho^{\varepsilon,\alpha}(Y_{1}^{\varepsilon},Y_{2}^{\varepsilon
})+2\varepsilon\rho^{\varepsilon,\alpha}(Y_{2}^{\varepsilon}%
,Y_{1}^{\varepsilon})/\alpha} dW_{1}\\
dY_{2}^{\varepsilon}=-\nabla V(Y_{2}^{\varepsilon})dt+\sqrt{2\varepsilon
\rho^{\varepsilon,\alpha}(Y_{1}^{\varepsilon},Y_{2}^{\varepsilon
})/\alpha+2\varepsilon\rho^{\varepsilon,\alpha}(Y_{2}^{\varepsilon}%
,Y_{1}^{\varepsilon})}dW_{2}%
\end{array}
,\right.
\]
and then define the weighted empirical measure of $(Y_{1}^{\varepsilon}%
,Y_{2}^{\varepsilon})$ and its permutation $(Y_{2}^{\varepsilon}%
,Y_{1}^{\varepsilon})$ by
\[
\zeta^{\varepsilon,T}(dx)\doteq\frac{1}{T}\int_{0}^{T}\left[  \rho
^{\varepsilon,\alpha}(Y_{1}^{\varepsilon},Y_{2}^{\varepsilon})\delta
_{(Y_{1}^{\varepsilon},Y_{2}^{\varepsilon})}(dx)+\rho^{\varepsilon,\alpha
}(Y_{2}^{\varepsilon},Y_{1}^{\varepsilon})\delta_{(Y_{2}^{\varepsilon}%
,Y_{1}^{\varepsilon})}(dx)\right]  dt,
\]
where%
\[
\rho^{\varepsilon,\alpha}(x_{1},x_{2})
=\frac{e^{-\frac{1}{\varepsilon}\left[  V(x_{1})+\alpha V(x_{2})\right]  }}{e^{-\frac{1}{\varepsilon}\left[  V(x_{1})+\alpha V(x_{2})\right]} +e^{-\frac{1}{\varepsilon}\left[  V(x_{2})+\alpha V(x_{1})\right] }},\quad
\]
(note that $\rho^{\varepsilon,\alpha}(x_{1},x_{2})+\rho^{\varepsilon,\alpha
}(x_{2},x_{1})=1$). One can show that $\zeta^{\varepsilon,T}(dx)$ has
precisely the same distribution as what one would obtain by forming the
ordinary empirical measure of the parallel tempering process with swap rate $a$  and letting $a\rightarrow\infty$.

\begin{remark}
We see that the infinite swapping scheme uses a \textit{symmetrized} version
of the original dynamics together with a \textit{weighted empirical measure}
to construct approximations to $\mu^{\varepsilon}(dx_{1})\mu^{\varepsilon
/\alpha}(dx_{2})$. As noted previously, the weights $\rho^{\varepsilon,\alpha
}$ will play an important role in the reduction of variance, and are in some
sense analogous to the likelihood ratio appearing in importance sampling
\cite{dupwusna}.
\end{remark}


\begin{remark}
Infinite swapping algorithms for continuous time reversible jump Markov
processes and for discrete time reversible Markov processes are also discussed
in \cite{dupliupladol,doldupnyq}.
\end{remark}


\subsection{Infinite swapping}

\label{subsubsec:infinite_swapping}

In this subsection we introduce the $K$-temperature INS estimator, which is
the main object of study. We use the following notation: $\boldsymbol{x}%
\doteq(x_{1},\ldots,x_{K})$ denotes an element in $M^{K}$; for any permutation $\sigma\in\Sigma_{K}$
and $\boldsymbol{x}\in M^{K}$, $\boldsymbol{x}_{\sigma}$ denotes
$(x_{\sigma(1)},\ldots,x_{\sigma(K)})$;
\[
\Delta\doteq\left\{  (x_{1},\ldots,x_{K})\in\mathbb{R}^{K}:1=x_{1}\geq
x_{2}\geq\cdots\geq x_{K}>0\right\}  ;
\]
$\boldsymbol{\alpha}\doteq(\alpha_{1},\ldots,\alpha_{K})$ $\in\Delta$ denotes
the $K$ temperature multiplication factors appearing in the definition of the
$K$-temperature INS estimator.

To define the $K$-temperature INS estimator for a given $\boldsymbol{\alpha,}$
we consider the (symmetric) diffusion process $\{\boldsymbol{X}^{\varepsilon
}(t)\}_{t\geq0}=\{(X_{1}^{\varepsilon}(t),\ldots,X_{K}^{\varepsilon
}(t))\}_{t\geq0}$ on $M^{K}$ satisfying%
\begin{equation}
\left\{
\begin{array}
[c]{l}%
dX_{1}^{\varepsilon}=-\nabla V\left(  X_{1}^{\varepsilon}\right)
dt+\sqrt{2\varepsilon}\sqrt{\rho_{11}^{\varepsilon}/\alpha_{1}+\rho
_{12}^{\varepsilon}/\alpha_{2}+\cdots+\rho_{1K}^{\varepsilon}/\alpha_{K}%
}dW_{1}\\
dX_{2}^{\varepsilon}=-\nabla V\left(  X_{2}^{\varepsilon}\right)
dt+\sqrt{2\varepsilon}\sqrt{\rho_{21}^{\varepsilon}/\alpha_{1}+\rho
_{22}^{\varepsilon}/\alpha_{2}+\cdots+\rho_{2K}^{\varepsilon}/\alpha_{K}%
}dW_{2}\\
\multicolumn{1}{c}{\vdots}\\
dX_{{\small K}}^{\varepsilon}=-\nabla V\left(  X_{K}^{\varepsilon}\right)
dt+\sqrt{2\varepsilon}\sqrt{\rho_{K1}^{\varepsilon}/\alpha_{1}+\rho
_{K2}^{\varepsilon}/\alpha_{2}+\cdots+\rho_{KK}^{\varepsilon}/\alpha_{K}%
}dW_{K}%
\end{array}
\right.  , \label{eqn:dym_INS}%
\end{equation}
where $W_{1},\ldots,W_{K}$ are independent Wiener processes and, for any
$i,j\in\{1,\ldots,K\}$ and $\sigma\in\Sigma_{K},$ $\rho_{ij}^{\varepsilon}$
denotes $\rho_{ij}^{\varepsilon}(\boldsymbol{X}^{\varepsilon}%
(t);\boldsymbol{\alpha})$ with
\[
\rho_{ij}^{\varepsilon}\left(  \boldsymbol{x};\boldsymbol{\alpha}\right)
\doteq\sum\limits_{\sigma:\sigma\left(  j\right)  =i}w^{\varepsilon}\left(
\boldsymbol{x}_{\sigma};\boldsymbol{\alpha}\right)  ,
\]
and with
\begin{equation}
w^{\varepsilon}\left(  \boldsymbol{x};\boldsymbol{\alpha}\right)  \doteq
\frac{\exp[-\frac{1}{\varepsilon}\sum_{\ell=1}^{K}\alpha_{\ell}V\left(
x_{\ell}\right)  ]}{\sum_{\sigma\in\Sigma_{K}}\exp[-\frac{1}{\varepsilon}%
\sum_{\ell=1}^{K}\alpha_{\ell}V\left(  x_{\sigma\left(  \ell\right)  }\right)
]}. \label{eqn:defofw}%
\end{equation}

Using detailed balance, one can show that for each $\varepsilon\in(0,\infty),$
$\nu^{\varepsilon}$ is the unique stationary distribution of $\{\boldsymbol{X}%
^{\varepsilon}(t)\}_{t\geq0},$ where%
\begin{equation}
\nu^{\varepsilon}\left(  d\boldsymbol{x}\right)  \doteq\frac{1}{K!Z_{\nu
}^{\varepsilon}}\sum_{\sigma\in\Sigma_{K}}\exp\left[  -\frac{1}{\varepsilon
}\sum\nolimits_{\ell=1}^{K}\alpha_{\ell}V\left(  x_{\sigma(\ell)}\right)
\right]  d\boldsymbol{x} \label{eqn:stat_dis_sym}%
\end{equation}
with%
\[
Z_{\nu}^{\varepsilon}\doteq\int_{M^{K}}\exp\left[  -\frac{1}{\varepsilon}%
\sum\nolimits_{\ell=1}^{K}\alpha_{\ell}V\left(  x_{\ell}\right)  \right]
d\boldsymbol{x}.
\]

\begin{remark}
For any $\sigma\in\Sigma_{K}$, we also have
\[
Z_{\nu}^{\varepsilon}=\int_{M^{K}}\exp\left[  -\frac{1}{\varepsilon}%
\sum\nolimits_{\ell=1}^{K}\alpha_{\ell}V\left(  x_{\sigma(\ell)}\right)
\right]  d\boldsymbol{x}.
\]

\end{remark}

Let $\zeta^{\varepsilon,T}\left(  d\boldsymbol{x}\right)  $ be the weighted
empirical measure of $\{\boldsymbol{X}^{\varepsilon}(t)\}_{t\geq0}$ over the
time interval of length $T$ given by%

\[
\zeta^{\varepsilon,T}\left(  d\boldsymbol{x}\right)  \doteq\frac{1}{T}\int
_{0}^{T}\sum_{\sigma\in\Sigma_{K}}w^{\varepsilon}\left(  \boldsymbol{X}%
_{\sigma}^{\varepsilon}\left(  t\right)  ;\boldsymbol{\alpha}\right)
\delta_{\boldsymbol{X}_{\sigma}^{\varepsilon}(t)}(d\boldsymbol{x})dt.
\]
It then follows from the ergodic theorem that $\zeta^{\varepsilon,T}$
converges in the topology of weak convergence of probability measures (and in
fact in the stronger $\tau$-topology \cite{dea4}) to $\mu^{\varepsilon
/\alpha_{1}}\times\mu^{\varepsilon/\alpha_{2}}\times\cdots\times
\mu^{\varepsilon/\alpha_{K}}$ w.p.1 as $T\rightarrow\infty.$ The
$K$-temperature INS estimator of $\mu^{\varepsilon}(A)$ with parameter
$\boldsymbol{\alpha}$ over time $T$ is therefore defined by
\begin{align}
\theta_{\text{INS}}^{\varepsilon,T}  &  \doteq\zeta^{\varepsilon,T}(A\times
M^{K-1})\label{eqn:INS_estimator}\\
&  =\frac{1}{T}\int_{0}^{T}\sum_{\sigma\in\Sigma_{K}}w^{\varepsilon}\left(
\boldsymbol{X}_{\sigma}^{\varepsilon}\left(  t\right)  ;\boldsymbol{\alpha
}\right)  1_{A}\left(  X_{\sigma\left(  1\right)  }^{\varepsilon}(t)\right)
dt.\nonumber
\end{align}

\begin{remark}
\label{rmk:quants}Besides $\mu^{\varepsilon}\left(  A\right)  $ for various
choices of $A$, one is also interested in estimating risk sensitive
functionals of the form%
\[
\int_{\mathbb{R}^{d}}e^{-\frac{1}{\varepsilon}F\left(  x\right)  }%
\mu^{\varepsilon}\left(  dx\right)  ,
\]
as well as the analogous integrals with respect to some or all of the higher
temperatures $\varepsilon/\alpha_{\ell}$. However, it is the lowest
temperature which is most challenging, and thus we focus on the problem of
estimating $\mu^{\varepsilon}(A)$ but seek rates of decay for the relative
error that are in some sense uniform in $A$.
\end{remark}

Before discussing a property which makes it heuristically clear why one would
expect $\theta_{\text{INS}}^{\varepsilon,T}$ to do better than $\theta
_{\text{MC}}^{\varepsilon,T},$ we introduce the notion of \textit{\ implied
potential}.

\begin{definition}
\label{Def:1.1}Given a probability density $\phi^{\varepsilon}$ with respect
to Lebesgue measure, we define the \textbf{implied potential} of
$\phi^{\varepsilon}$ to be $-\varepsilon\log\phi^{\varepsilon}.$
\end{definition}

\begin{example}
\label{Ex:1.1}If $\mu^{\varepsilon}$ is a Gibbs measure as in (\ref{eqn:Gibbs}%
), then the implied potential of $\mu^{\varepsilon}$ is $V$, the potential
appearing in the dynamics (\ref{eqn:dym_original}).
\end{example}

From Example \ref{Ex:1.1} we see that implied potential generalizes the notion
of potential. By comparing the implied potential of $\nu^{\varepsilon}$ as in
(\ref{eqn:stat_dis_sym}) and the product measure $\mu^{\varepsilon/\alpha_{1}%
}\times\cdots\times\mu^{\varepsilon/\alpha_{K}}$ with $\mu^{\varepsilon}$ as
in (\ref{eqn:Gibbs}), one can show that the maximum barrier of the implied
potential of the former is smaller than that of the latter provided that
$\alpha_{\ell}<1$ for some $\ell\in\{2,\ldots,K\}$. Since as is well known the
barrier heights determine the exponential time scale of transitions between
neighborhoods of local minimum of the implied potential, this lowering of the
energy barriers is expected to enhance the sampling of the entire space.

While it is intuitive that lowering energy barriers is helpful, it does not by
itself lead to schemes that are in any sense optimal at low temperatures. A
more important and open question in the design of the $K$-temperature INS
estimator is how to select the ensemble of multiplicative factors
$\boldsymbol{\alpha}$. In this paper we not only characterize the low
temperature performance of a $K$-temperature INS estimator with a fixed set of
temperature factors $\boldsymbol{\alpha}$, but we also provide optimal and
nearly optimal temperatures for problems of interest in the same limit. As we
will see, the optimal temperature schedule is dominated by a geometric
relation, and moreover is fairly insensitive to the particular numerical
quantity of interest.

\subsection{Performance measure}

\label{subsec:performance_measure}

In this subsection we discuss the performance measure that will be used to
characterize good performance of an estimator. Let $\{\bar{X}^{\varepsilon
}\}_{\varepsilon\in(0,\infty)}\subset C([0,T];\bar{M})$ be a sequence of
stochastic processes that will be used to define an estimator. For complicated
potentials $V$ we expect these processes to exhibit metastability, which means
that the time required for $\bar{X}^{\varepsilon}$ to visit the various\ parts
of the state space that are needed for good estimation scales like
$T^{\varepsilon}=e^{\frac{1}{\varepsilon}c}$ for some $c>0$. As a consequence,
if we wish to compare algorithms after they have become reasonably accurate we
should assume the simulation interval scales in this way.

As noted in Remark \ref{rmk:quants}, we focus on the problem of estimating
$\mu^{\varepsilon}(A)$ for some set $A\subset M$, and assume there is a
large deviation limit (i.e., $\lim_{\varepsilon\rightarrow0}\varepsilon\log
\mu^{\varepsilon}(A)$ exists).

\begin{definition}
\label{Def:1.8}An estimator $\theta^{\varepsilon,T^{\varepsilon}}$ of
$\mu^{\varepsilon}(A)$ is called \textbf{essentially unbiased}  if there is
$c_{0}\in(0,\infty)$ such that for any $x\in\bar{M}$
\[
\liminf_{\varepsilon\rightarrow0}-\varepsilon\log\left(  \left\vert
E_{x}\theta^{\varepsilon,T^{\varepsilon}}-\mu^{\varepsilon}(A)\right\vert
\right)  \geq\lim_{\varepsilon\rightarrow0}-\varepsilon\log\mu^{\varepsilon
}(A)+c_{0}.
\]

\end{definition}

This says that the bias of $\theta^{\varepsilon,T^{\varepsilon}}$ (i.e., the
difference between $E_{x}\theta^{\varepsilon,T^{\varepsilon}}$ and
$\mu^{\varepsilon}(A)$) decays strictly faster than $\mu^{\varepsilon}(A)$ as
$\varepsilon\rightarrow0$.

\begin{definition}
\label{Def:1.9}Given an estimator $\theta^{\varepsilon,T^{\varepsilon}},$ the
\textbf{lower bound on the decay rate of the variance per unit time} of
$\theta^{\varepsilon,T^{\varepsilon}}$ is defined as
\[
\inf_{x\in\bar{M}}\liminf_{\varepsilon\rightarrow0}-\varepsilon\log\left(
\mathrm{Var}_{x}\left(  \theta^{\varepsilon,T^{\varepsilon}}\right)
T^{\varepsilon}\right)  .
\]
If the $\liminf$ is a limit that does not depend on $x$, then we call it
\textbf{the decay rate of the variance per unit time.}
\end{definition}

\begin{remark}
\label{Rmk:1.2}In this paper, we seek to optimize the {decay rate of the
variance per unit time} (often referred to simply as the decay rate of the
variance), but only among estimators that are essentially unbiased. A
criticism is that essential unbiasedness depends on the time scaling
$T^{\varepsilon}$, which may itself depend on the estimator. One may be
concerned that improving the decay rate somehow lengthens the time till
essential unbiasedness. However, as we discuss in a moment the selection of
INS temperature parameters that lower the decay rate of the variance also
reduce the growth rate of this time. Thus there is no conflict in using the
decay rate of the variance as the sole performance measure.
\end{remark}

\begin{remark}
\label{rmk:benchmark}We will take as our ideal performance benchmark a decay
rate of the variance exactly \textit{twice} $\lim_{\varepsilon\rightarrow
0}-\varepsilon\log\mu^{\varepsilon}(A)$. The reason is as follows. Suppose
that we measure errors by the standard deviation (and assume essential
unbiasedness). If we achieve this best possible decay rate, then the amount of
time needed for the numerical error $\theta^{\varepsilon,T^{\varepsilon}}%
-\mu^{\varepsilon}(A)$ to be comparable to $\mu^{\varepsilon}(A)$ itself
becomes subexponential in $\varepsilon$. See Remark \ref{rmk:bench2} for a
more detailed statement.

Strictly speaking, $2\lim_{\varepsilon\rightarrow0}-\varepsilon\log
\mu^{\varepsilon}(A)$ is not the best possible decay rate of the variance, but
rather the best practically achievable decay rate. Indeed, in analogy with the
zero variance estimator that one can define when using importance sampling for
rare event estimation \cite{buddup4, blalam}, it is possible to define
estimators with a larger decay rate. But these are not useful since they
require information that is not typically available, such as knowing
$\mu^{\varepsilon}(A)$. \textit{Hence the aim in the design of an INS
algorithm is to obtain a lower bound on the decay rate of the variance that is
close to this maximum practical value.}
\end{remark}

\section{Statement of the Main Results}

\label{sec:statement_of_the_main_results}

In this section we state the main results on the performance and optimal
design of the INS scheme in the low temperature limit. The proofs involve
applying the results of \cite{dupwu}, and then simplifying the variational
problem that characterizes the decay rate of the variance.

We present two main results. The first considers the restricted setting of a
simple two well model. In this case we can obtain a very precise reduction of
the variational problem. Using this simplified expression, we can then
probe in some detail the question of how INS achieves variance reduction. Our
interest in this model is twofold. One reason is that with an exact expression
(rather than a tight bound) for the solution to the variational problem we can
explore issues relating to how variance reduction is obtained through
swapping. The second is that it properly suggests very useful bounds for the
general model. (While exact simplifications are possible there as well, the
number of cases quickly becomes unwieldy as the number of local minima
increases.) Since the proof of the reduction is long, we refer the reader to
\cite{wu2} for details.

The second main result is concerned with temperature selection when there are
an arbitrary number of wells. Owing to this generality, we do not attempt to
find the exact optimizer, but rather show that the geometric relation for
temperatures suggested by the two well model allows one to get arbitrarily
close to the benchmark articulated in Remark \ref{rmk:benchmark}, with the
\textquotedblleft gap\textquotedblright\ between the two taking the form $(1/2)^{K-1}(V(A)+B)$ for some positive constant $B$, and therefore decaying
geometrically in the number of temperatures. The proof of this result is also somewhat
detailed, and is started in Section \ref{sec:lower_bound} and completed in Section \ref{sec:6}.
In particular,
the results of Section \ref{sec:6} show how $B$ depends in a natural and intuitive way on properties of the original potential $V$.

To apply the results of \cite{dupwu} we need to know that the INS process
defined in (\ref{eqn:dym_INS}) satisfies a large deviation principle on
$C([0,T]:M^{K})$ for arbitrary $T\in(0,\infty)$. This is not straightforward,
owing to the fact that the diffusion coefficients involve $w^{\varepsilon
}\left(  \boldsymbol{x};\boldsymbol{\alpha}\right)  $ defined in
(\ref{eqn:defofw}), which become discontinuous in $\boldsymbol{x}$ as
$\varepsilon\rightarrow0$. Hence one is concerned with the large deviation
properties of processes with \textit{discontinuous statistics}
\cite{dupellwei,dupell2}.

The sorts of discontinuities encountered are in fact analogous to those
encountered in the large deviation analysis of stochastic networks, such as
multiclass queueing networks. A general approach to proving that a large
deviation principle holds for stochastic networks appears in \cite{dupell2},
and can be adapted to the INS model (\ref{eqn:dym_INS}). It is important to
note that we do not need the precise form of the rate function, but only that
the LDP holds with some rate function and basic qualitative properties. This
is because with the INS model we already have an expression for the stationary
distribution. Various quantities are defined in \cite{dupwu} using the rate
function that allow the identification of the Freidlin-Wentzell quasipotential
and related objects. For the INS model the explicit formula for the stationary
distribution directly identifies the quasipotential, thereby eliminating the
need for the explicit form of the rate function.
The technique of \cite{dupell2} is in fact
ideally suited to showing the existence of an LDP without necessarily having
an expression for the rate function. We will assume the needed existence of
the large deviation principle, and outline in the Appendix how one can adapt
the argument of \cite{dupell2} to (\ref{eqn:dym_INS}) for the case of the two
well model, which features the main issues in the proof of the general case.

\subsection{Two well model}

Our first result considers the setting of a double-well potential. Let
$V:\mathbb{R}\rightarrow\mathbb{R}$ ($d=1$) be as in the following figure.

\begin{figure}[h]
    \centering
    \includegraphics[width=0.9\textwidth]{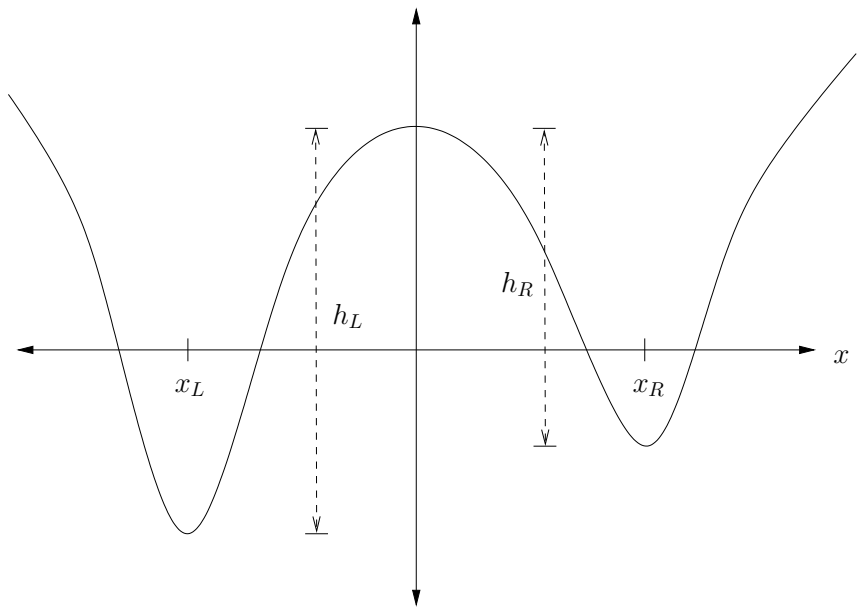}
    \caption{Asymmetric two well model}
    \label{fig:well}
\end{figure}

Assume $V$ satisfies the following condition.

\begin{condition}
\label{Con:2.1}$V$ is a bounded $C^{2}$ function and

\begin{itemize}
\item $V$ is defined on a compact interval $D\subset\mathbb{R}$ and extended
periodically as a $C^{2}$ function.

\item $V$ has only two local minima at $x_{L}$ and $x_{R}$ with values
$V(x_{L})<V(x_{R})$.

\item $V$ has only one local maximum at $0\in(x_{L},x_{R})$.

\item $V(x_{L})=0,$ $V(0)=h_{L}$ and $V(x_{R})=h_{L}-h_{R}>0.$

\item $\inf_{x\in\partial D}V(x)>h_{L}.$
\end{itemize}
\end{condition}

\begin{remark}
As noted previously, the use of periodic boundary conditions is common in
numerical implementation. It is assumed that the boundary is away from the
neighborhoods of the equilibrium points of interest, and that the potential at
the boundary is high enough that transitions across the boundary are
unimportant. For our purposes, this means that the relevant large deviation
calculations involve only paths that remain in $D$.
\end{remark}

\begin{remark}
\label{rmk:time_horizon} In the analysis of $\theta_{\text{INS}}%
^{\varepsilon,T^{\varepsilon}}$ we will assume $T^{\varepsilon}$ satisfies
$T^{\varepsilon}=e^{\frac{1}{\varepsilon}c}$ with $c>\alpha_{K}h_{L}$. Recall
that $\alpha_{K}$ is the smallest of the $\alpha_{\ell}$, and hence determines
the highest temperature. As we will see, this condition ensures asymptotic unbiasedness.
\end{remark}

The next result follows from \cite[Theorems 4.3 and 4.5]{dupwu}. The theorem, in
particular, characterizes the decay rate of the variance for the INS estimator
for a given $\boldsymbol{\alpha}$. For a set $A$ let $V(A)\doteq\inf_{x\in
A}V(x)$, and also define $\mathcal{K}\doteq\{1,2,\ldots,K+1\}$.

\begin{theorem}
\label{Thm:2.1}Assume Condition \ref{Con:2.1}, and that the process defined by
(\ref{eqn:dym_INS}) satisfies a large deviation principle that is uniform with
respect to initial conditions \cite[Section 1.2]{buddup4}. Then for any closed
interval $A\subset D$ with $x_{L}\notin A$ and $A=\bar{A}^{\circ},$
\begin{equation}
\theta_{\mathrm{INS}}^{\varepsilon,T^{\varepsilon}}=\frac{1}{T^{\varepsilon}%
}\int_{0}^{T^{\varepsilon}}\sum_{\sigma\in\Sigma_{K}}w^{\varepsilon}\left(
\boldsymbol{X}_{\sigma}^{\varepsilon}\left(  t\right)  ;\boldsymbol{\alpha
}\right)  1_{A}\left(  X_{\sigma\left(  1\right)  }^{\varepsilon}(t)\right)
dt \label{eqn:defoftheps}%
\end{equation}
is an essentially unbiased estimator of $\mu^{\varepsilon}(A),$ where
$w^{\varepsilon}\left(  \boldsymbol{x};\boldsymbol{\alpha}\right)  $ is given
by (\ref{eqn:defofw}). Moreover, for any   $\boldsymbol{\alpha}\in\Delta
$ and $\boldsymbol{x} \in \mathbb{R}^K$, we have
\[
\liminf_{\varepsilon\rightarrow0}-\varepsilon\log\left(  \mathrm{Var}%
_{\boldsymbol{x}}\left(  \theta_{\mathrm{INS}}^{\varepsilon,T^{\varepsilon
}}\right)  T^{\varepsilon}\right)  \geq\left\{
\begin{array}
[c]{c}%
\hat{r}_{1}\left(  \boldsymbol{\alpha}\right)  \wedge\hat{r}_{3}\left(
\boldsymbol{\alpha}\right) ,\text{ if }A\subset(-\infty,0]\\
\hat{r}_{1}\left(  \boldsymbol{\alpha}\right)  \wedge\hat{r}_{2}\left(
\boldsymbol{\alpha}\right) ,\text{ if }A\subset\lbrack0,\infty)
\end{array}
,\right.
\]
where
\[
\hat{r}_{1}\left(  \boldsymbol{\alpha}\right)  \doteq\inf
\nolimits_{\boldsymbol{x}\in A\times%
\mathbb{R}
^{K-1}}\left[  2\sum_{\ell=1}^{K}\alpha_{\ell}V\left(  x_{\ell}\right)
-\min_{\sigma\in\Sigma_{K}}\left\{  \sum_{\ell=1}^{K}\alpha_{\ell}V\left(
x_{\sigma\left(  \ell\right)  }\right)  \right\}  \right]  ,
\]%
\[
\hat{r}_{2}\left(  \boldsymbol{\alpha}\right)  \doteq\min_{i\in\mathcal{K}%
\setminus\{1\}}\left\{  2V\left(  A\right)  +\left[  \sum_{\ell=1}^{i-2}%
\alpha_{K-\ell+1}-\alpha_{K-i+2}\right]  \left(  h_{L}-h_{R}\right)  \right\}
-\alpha_{K}h_{R},
\]
and
\[
\hat{r}_{3}\left(  \boldsymbol{\alpha}\right)  \doteq2V\left(  A\right)
-\alpha_{K}h_{L}.
\]
\begin{remark}
As mentioned in  \cite[Conjecture 4.10]{dupwu}, we expect that the lower bound is tight. The proof of the conjecture for a special case is outlined in \cite[Section 11]{dupwu}.
\end{remark}

\end{theorem}

Recall that the optimal decay rate of the variance per unit time is twice the large
deviation decay rate of $\mu^{\varepsilon}(A)$, which is $V(A)$. The next two
results identify optimizers over $\boldsymbol{\alpha}$ for the relevant
variational problems. Note that in all cases we can get close to the best
possible decay rate by choosing $K$ appropriately, and in fact the gap goes to
zero geometrically in $K$. For example, $K=7$ will to get within 2\% of the
maximum rate of $2V(A)$.

\begin{theorem}
\label{Thm:2.2}Assume the conditions of Theorem \ref{Thm:2.1}. For any closed
set $A\subset(-\infty,0]$ with $x_{L}\notin A,$ if $V(A)\geq h_{L},$ then%
\[
\sup_{\boldsymbol{\alpha}\in\Delta}\left[  \hat{r}_{1}\left(
\boldsymbol{\alpha}\right)  \wedge\hat{r}_{3}\left(  \boldsymbol{\alpha
}\right)  \right]  =2V\left(  A\right)  -\left(  1/2\right)  ^{K-1}V\left(
A\right)
\]
with the optimal $\boldsymbol{\alpha}^{\ast}=\left(  1,1/2,\ldots,\left(
1/2\right)  ^{K-2},\left(  1/2\right)  ^{K-1}\right)  \in\Delta.$ If $V(A)\leq
h_{L},$ then
\[
\sup_{\boldsymbol{\alpha}\in\Delta}\left[  \hat{r}_{1}\left(
\boldsymbol{\alpha}\right)  \wedge\hat{r}_{3}\left(  \boldsymbol{\alpha
}\right)  \right]  =2V\left(  A\right)  -\left(  1/2\right)  ^{K-2}\left(
\frac{h_{L}}{V(A)+h_{L}}\right)  V(A)
\]
with the optimal $\boldsymbol{\alpha}^{\ast}=\left(  1,1/2,\ldots,\left(
1/2\right)  ^{K-2},\frac{V(A)}{V(A)+h_{L}}\left(  1/2\right)  ^{K-2}\right)
\in\Delta.$
\end{theorem}

\begin{theorem}
\label{Thm:2.3}Assume the conditions of Theorem \ref{Thm:2.1}. For any closed
set $A\subset\lbrack0,\infty)$ and if $h_{L}\geq2h_{R}$ or $V(A)\geq h_{L}$,
then
\[
\sup_{\boldsymbol{\alpha}\in\Delta}\left[  \hat{r}_{1}\left(
\boldsymbol{\alpha}\right)  \wedge\hat{r}_{2}\left(  \boldsymbol{\alpha
}\right)  \right]  =2V\left(  A\right)  -\left(  1/2\right)  ^{K-1}\left(
V(A)\vee h_{L}\right)
\]
with the optimal $\boldsymbol{\alpha}^{\ast}=\left(  1,1/2,\ldots,\left(
1/2\right)  ^{K-2},\left(  1/2\right)  ^{K-1}\right)  \in\Delta.$ If
$h_{L}\leq2h_{R}$ and $V(A)\in\lbrack h_{L}-h_{R},h_{L}]$, then
\[
\sup_{\boldsymbol{\alpha}\in\Delta}\left[  \hat{r}_{1}\left(
\boldsymbol{\alpha}\right)  \wedge\hat{r}_{2}\left(  \boldsymbol{\alpha
}\right)  \right]  =2V\left(  A\right)  -\left(  1/2\right)  ^{K-2}\left(
\frac{h_{R}}{V(A)-(h_{L}-2h_{R})}\right)  V(A)
\]
with the optimal $\boldsymbol{\alpha}^{\ast}=\left(  1,1/2,\ldots,\left(
1/2\right)  ^{K-2},\frac{V(A)-(h_{L}-h_{R})}{V(A)-(h_{L}-2h_{R})}\left(
1/2\right)  ^{K-2}\right)  \in\Delta.$
\end{theorem}

\begin{remark}
\label{Rmk:2.1}According to Theorems \ref{Thm:2.1}, \ref{Thm:2.2} and
\ref{Thm:2.3}, no matter what set $A$ is considered, the optimal temperatures
$\boldsymbol{\alpha}^{\ast}$ form a geometric sequence with common ratio
$1/2$, except possibly the last and smallest value, which corresponds to the
highest temperature.


\end{remark}

\begin{remark}
\label{Rmk:2.2} By Theorems \ref{Thm:2.1} and \ref{Thm:2.3}, if $A\subset
[0,\infty)$, $h_{L}\leq2h_{R}$ and $V(A)=h_{L}-h_{R}$, the last component of
the optimal temperature $\boldsymbol{\alpha}^{\ast}$ is $0$. Of course the INS
estimator is not well-defined with $\alpha_{K}^{\ast}=0$. In fact,
$\boldsymbol{\alpha}^{\ast}$ is not in $\Delta$, though it is in the closure
of $\Delta$. However, since $\hat{r}_{1}\left(  \boldsymbol{\alpha}\right)
\wedge\hat{r}_{2}\left(  \boldsymbol{\alpha}\right)  $ is a continuous
function of $\boldsymbol{\alpha}$, we can always approach the optimal
performance by using $\boldsymbol{\alpha}$ which is close to
$\boldsymbol{\alpha}^{\ast}$, e.g., $\boldsymbol{\alpha}=( 1,1/2,\ldots
,\left(  1/2\right)  ^{K-2},\delta\left(  1/2\right)  ^{K-2} )$ for some
$\delta\in(0,1)$.
\end{remark}

\begin{remark}
Analogous results hold for a high-dimensional double-well potential
$V:\mathbb{R}^{d}\rightarrow\mathbb{R}$, where $x_{L}$ and $x_{R}$ are the two
local minima (and the former is the unique global minimum) and $0$ is the
unique local maximum. Moreover, one should interpret $(-\infty,0]$ and $[0,\infty)$
as the closure of the domain of attraction of $x_{L}$ and that of $x_{R}$, respectively.
\end{remark}

\begin{remark}
\label{rmk:bench2}
Let $\gamma_{i}>0,i=1,2$ be given. Suppose that the lower
bound on the variance decay rate is within $\gamma_{1}$ of the benchmark, here
$2V(A)$, and that also $\alpha_{K}h_{L}\leq(1/2)^{K-1}h_{L}<\gamma_{2}$.
When this is true, with the simulation time horizon $T^{\varepsilon}%
=e^{\frac{1}{\varepsilon}\gamma_{2}}$ (see Remark \ref{rmk:time_horizon}) we
find that for small $\varepsilon>0$
\[
\text{\textrm{Var}}_{x}\left(  \theta_{\text{\textrm{INS}}}^{\varepsilon
,T^{\varepsilon}}\right)  \leq\frac{1}{T^{\varepsilon}}e^{-\frac
{2}{\varepsilon}(V(A)-\gamma_{1})},
\]
while the quantity being estimated is (approximately) of magnitude
$\mu^{\varepsilon}(A)\approx e^{-\frac{1}{\varepsilon}V(A)}$. Therefore the
ratio of the standard deviation of the estimator (recall that the bias will be
negligible) to the quantity of interest satisfies%
\[
\frac{\text{\textrm{SD}}_{x}\left(  \theta_{\text{\textrm{INS}}}%
^{\varepsilon,T^{\varepsilon}}\right)  }{\mu^{\varepsilon}(A)}\leq
e^{\frac{1}{\varepsilon}(\gamma_{1}-\gamma_2/2)},
\]
with simulation time that scales like $e^{\frac{1}{\varepsilon}\gamma_{2}}$, and bounded relative error requires, in addition to the bound above, $\gamma_2 > 2 \gamma_1$. Although the simulation time grows exponentially in $1/\varepsilon$, the constant gets small very quickly as $K$ increases. Note also that the bound applies for arbitrary sets $A$. For comparison, 
let $\delta_1>0$ and $\delta_2>h_L$.
If we consider standard Monte Carlo with $T^\varepsilon=e^{\frac{1}{\varepsilon}\delta_2}$ we would have a lower bound of the form 
\[
\text{\textrm{Var}}_{x}\left(  \theta_{\text{\textrm{MC}}}^{\varepsilon
,T^{\varepsilon}}\right)  \geq\frac{1}{T^{\varepsilon}}e^{-\frac
{1}{\varepsilon}(V(A)+\delta_{1})},
\]
for small $\varepsilon>0$, and
\[
\frac{\text{\textrm{SD}}_{x}\left(   \theta_{\text{\textrm{MC}}}^{\varepsilon
,T^{\varepsilon}}\right)  }{\mu^{\varepsilon}(A)}\geq
e^{\frac{1}{\varepsilon}(V(A)/2-\delta_1/2 - \delta_2/2)}.
\]
In this case we cannot reduce $T^\varepsilon$ below $e^{\frac{1}{\varepsilon}h_L}$. If $V(A)<h_L$ we can have bounded relative error, but if the set is moved further to the right so that $V(A)>h_L$ then
we must increase the growth rate of $T^\varepsilon$ for bounded relative error.
In all cases, the time required grows exponentially in $1/\varepsilon$ and, unlike the INS case, we cannot make the constant small. 
\end{remark}

\subsection{Sources of variance reduction}

\label{subsec:source_of_variance_reduction} Here we make some remarks on the
form of the optimal $\boldsymbol{\alpha}$ and its interpretation regarding how
variance reduction is achieved by INS. The remarks will also apply to parallel
tempering to some extent if the swap rate is sufficiently high, though in
this case the weights $\rho$ used in INS are then implicitly computed by the
algorithm, giving another sense in which INS is an optimized version of PT.

To begin, we note that the most obvious qualitative change when adding a
higher temperature particle to one or more particles with lower temperature is
that the \textquotedblleft mobility,\textquotedblright\ by which we mean the
ease with which it crosses energy barriers, of the new particle is greater
than that of all other particles. (What this means for INS is that the
particle with the currently highest value of $V$ is essentially given this
temperature, with a slightly modified interpretation when two or more
particles share the highest $V$ value.)

Hence it is tempting to explain the improved sampling of INS, especially with
respect to functionals that correspond to integration with respect to the
lowest temperature, as a consequence of this greater mobility being passed
between higher temperatures and lower temperatures. The mobility is passed via
the swap mechanism with PT, and by the $\rho$ weights with INS. For example,
with PT the argument would be that the sharing of mobility between different
temperatures obtained via swapping makes it easier for the low temperature
particle to overcome potential barriers, and hence the empirical measure will
converge more quickly. While plausible in a qualitative way, it is not clear,
for example, how to relate the claim of faster convergence of the empirical
measure to the properties of the variance. In fact, the situation is more complex.

In order to understand the role played by \textquotedblleft
mobility,\textquotedblright\ in a previous paper \cite{dupwusna} we introduced
and studied what we call INS for IID, which stands for \textit{infinite
swapping} for \textit{independent and identically distributed random variables}. The
setting of that paper considers the integral of a distribution with respect to
some risk-sensitive functional (including as a special case probabilities of
sets with a positive large deviation rate, as is the case of Theorems
\ref{Thm:2.1}, \ref{Thm:2.2} and \ref{Thm:2.3}). Because straightforward Monte
Carlo will not work well, the paper follows the logic of parallel tempering
but within the context of INS. It is assumed the distribution (say
$\mu_{\varepsilon}$) is indexed by a parameter $\varepsilon$ that corresponds
to temperature here, and that a large deviation principle holds for
$\{\mu_{\varepsilon}\}$ with a known rate function. This
measure is then coupled with measures indexed by higher values of the temperature using a
parameter exactly analogous to $\boldsymbol{\alpha}$, and using symmetrization
in the same way as INS one can define an estimator for integrals with respect
to the lowest temperature using $\rho$ weights in the way (suitable for the
static setting) that is exactly analogous to what is done in the present paper
for the Markov setting. Knowledge of the LD rate function is what allows for
the explicit computation of the analogues of the $\rho$ weights. This produced
unbiased estimators analogous to those of the Markov setting, but for this
purely static setting.

A key observation is the following. Since the setting of \cite{dupwusna} does
not involve any dynamics, the notion that any variance reduction is due to
\textquotedblleft increased mobility\textquotedblright\ is not possible.
Indeed, as is discussed in \cite{dupwusna} the $\rho$ weights act in a way
similar to the likelihood ratio in a well designed importance sampling scheme,
helping to cluster the values of the unbiased estimate around the true value,
thereby reducing variance. We argue that the analogous property holds here,
and that the primary role of the higher temperatures (except possibly the
highest temperature) is to provides this variance reduction, and that solving
the variational problems as in Theorems \ref{Thm:2.2} and \ref{Thm:2.3} tells
us how to do this in the low temperature limit. Indeed, we obtain exactly the
same geometric spacing of all temperatures (save the highest) in the low
temperature limit in the Markovian setting as was obtained in the static
setting. An analogous claim could be made regarding PT in the high swap rate
setting, though as noted for PT the computation of the weights is carried out
implicitly via the swaps and averaging in time.

While this motivates the form of the lower temperatures, it leaves out the
highest temperature. Here we find a variety of behaviors that depend on the
particular quantity that is being estimated, and one might argue that it is
here that the mobility of a particle plays a role in determining the value of
$\alpha_{K}$. In all the cases of Theorems \ref{Thm:2.2} and \ref{Thm:2.3}, we
find that the optimal $\alpha_{K}$ is less than or equal to $(1/2)^{K-1}$,
which is the value one finds in the static setting. We conjecture that the
perturbation of $\alpha_{K}$ away from $(1/2)^{K-1}$ in the Markov case
reflects that the optimization here benefits more from greater mobility than
the variance reducing effects of the geometric sequence. There is even one
case, where the optimal value of $\alpha_{K}$ is zero, which one can interpret
as saying we should make the corresponding component as noisy as possible. It
is also worth noting that the overall performance is not particularly
sensitive to $\alpha_{K}$ having the optimal value, in that if we were to
simply use the purely geometric sequence then we still have a decay rate that
is within $(1/2)^{K-1}(V(A)\vee h_L)$ of the optimal $2V(A)$. 

\subsection{Multiple well model}

\label{sec:lower_bound_with_geometric_spacing} The second main result
considers a finite but otherwise arbitrary number of wells. While it is
possible that one could derive results analogous to Theorems \ref{Thm:2.2} and
\ref{Thm:2.3}, which identify the optimizer appearing in the lower bound of
Theorem \ref{Thm:2.1}, we will instead settle for showing that the geometric
spacing suggested by the two well model leads to a variance decay rate that
can be made close to the optimum of $2V(A)$. The parameter $B$ that appears in
the following theorem depends only on $V$, and is identified in Remark
\ref{rem:defofB}. In particular, it does not depend on $\varepsilon
$.
As will be illustrated by examples in Section \ref{sec:6},
$B$ contains interesting information on how the geometry and other properties of the original potential $V$ affect the rate of decay of the variance. 
For example, if the well that corresponds to the global minimum $O_1$ is also the most difficult well to escape from, then the situation of the multiple well model is very similar to that of the two well model.
However, when this is not the case one can have $B>V(A)$,
and $B$ will depend on how the local minima are interconnected. 

For the following theorem, we assume that $V:$ $M\rightarrow\mathbb{R}$ is a smooth
multi-well potential  with a unique global minimum
$y_{1}\in M$ and without loss normalize so that $V$ takes value $0$ at $y_{1}$ (i.e., $V(y_{1})=0$ and
$V(x)>0$ for all $x\in M$). We assume that the gradient of $V$ is Lipschitz continuous,
and also assume that there exists a finite collection of points $\{O_{i}%
\}_{i\in L}\subset M^{K}$ with $L\doteq\{1,2,\ldots,l\}$ for some
$l\in\mathbb{N}$, such that $\cup_{i\in L}\{O_{i}\}$ coincides with the
$\omega$-limit set of the zero noise analogue of \eqref{eqn:dym_INS}, so that $\cup_{i\in L}\{O_{i}\}=\{y_1,\ldots,y_H\}^K$.
This
imposes some additional structure on $V$, and in particular rules out open
regions on which $V$ is a constant.

\begin{theorem}
\label{thm:lower_bound_INS} Assume that the process defined by (\ref{eqn:dym_INS})
satisfies a large deviation principle that is uniform with respect to initial
conditions. 
Then there exists $B<\infty$ such that the following hold.
Consider the choice $\boldsymbol{\alpha}^{\ast}=(1,1/2,\dots
,(1/2)^{K-1})$ and let $T^{\varepsilon}=e^{\frac{1}{\varepsilon}c}$ for some
$c>B\alpha^{\ast}_{K}=B\left(  1/2\right)  ^{K-1}$. Define $\theta_{\mathrm{INS}%
}^{\varepsilon,T^{\varepsilon}}$ by (\ref{eqn:defoftheps}). Then
$\theta_{\mathrm{INS}}^{\varepsilon,T^{\varepsilon}}$ is essentially
unbiased, and
\[
\liminf_{\varepsilon\rightarrow0}-\varepsilon\log\left(  \mathrm{Var}%
_{x}\left(  \theta_{\mathrm{INS}}^{\varepsilon,T^{\varepsilon}}\right)
T^{\varepsilon}\right)  \geq\left(  2-\left(  1/2\right)  ^{K-1}\right)
V\left(  A\right)  -B\left(  1/2\right)  ^{K-1}.
\]

\end{theorem}

\section{Proof of Theorem \ref{thm:lower_bound_INS}}

\label{sec:lower_bound}

We first recall notation from Subsection \ref{subsubsec:infinite_swapping} and
introduce additional notation. Given $K\in\mathbb{N},$ for any
$\boldsymbol{\alpha}\in\Delta$ we consider the diffusion process
$\{\boldsymbol{X}^{\varepsilon}(t)\}_{t\geq0}=\{(X_{1}^{\varepsilon}%
(t),\ldots,X_{K}^{\varepsilon}(t))\}_{t\geq0}$ on $M^{K}$ satisfying
\eqref{eqn:dym_INS}, 
and denote $O_{1}%
\doteq(y_{1},\dots,y_{1})$. Figure \ref{fig:b1} illustrates the points
$\cup_{i\in L}\{O_{i}\}$ when $V$ is the Franz potential and $K=2$, with
$O_{1},O_{3},O_{7}$ and $O_{9}$ local minima in the multidimensional potential
defined in (\ref{eqn:U}), $O_{2},O_{4},O_{6}$ and $O_{8}$ saddle points, and
$O_{5}$ a local maximum.


\begin{figure}[h]
    \centering
    {{\includegraphics[width=5cm]{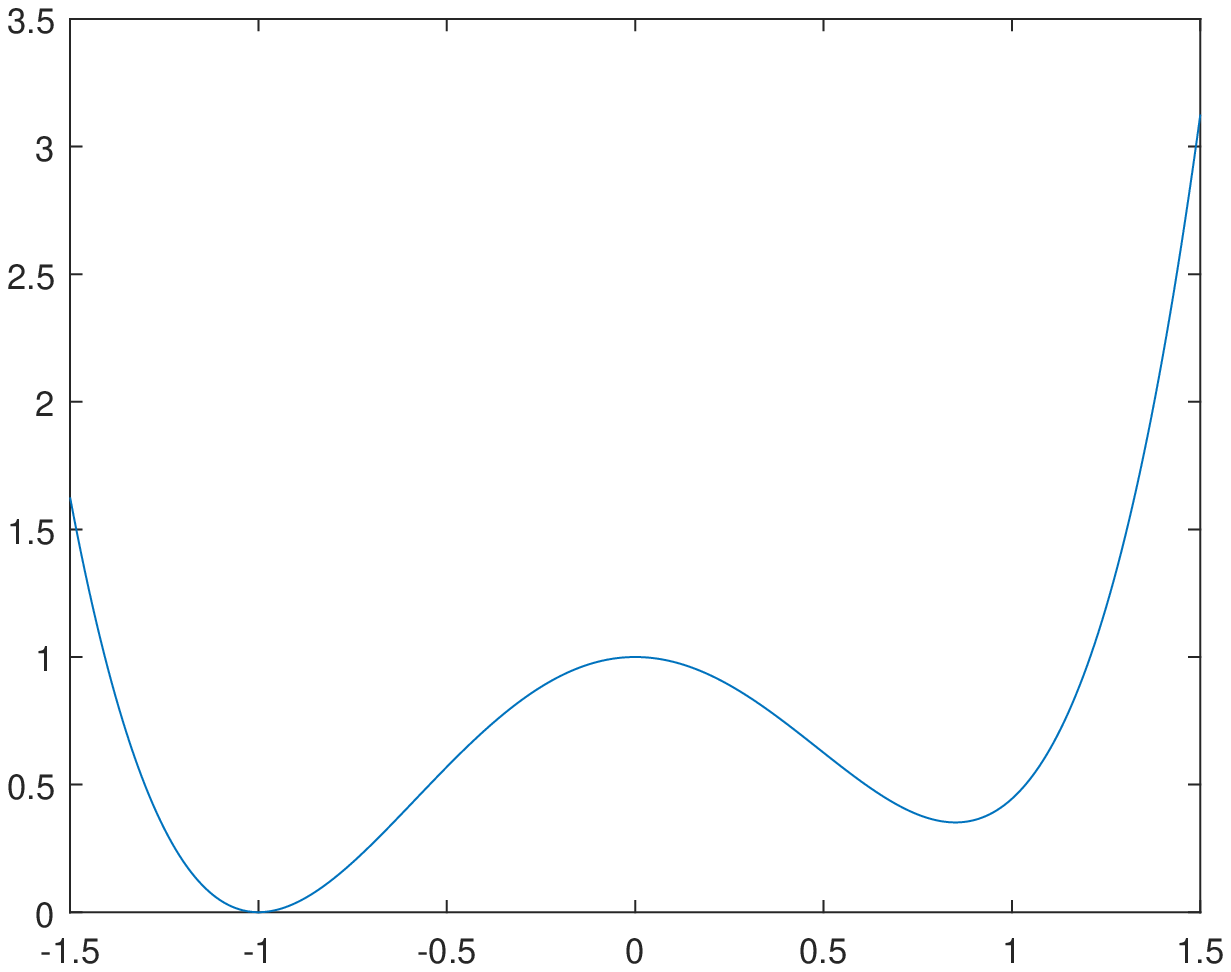} }}%
    \qquad
   {{\includegraphics[width=5cm]{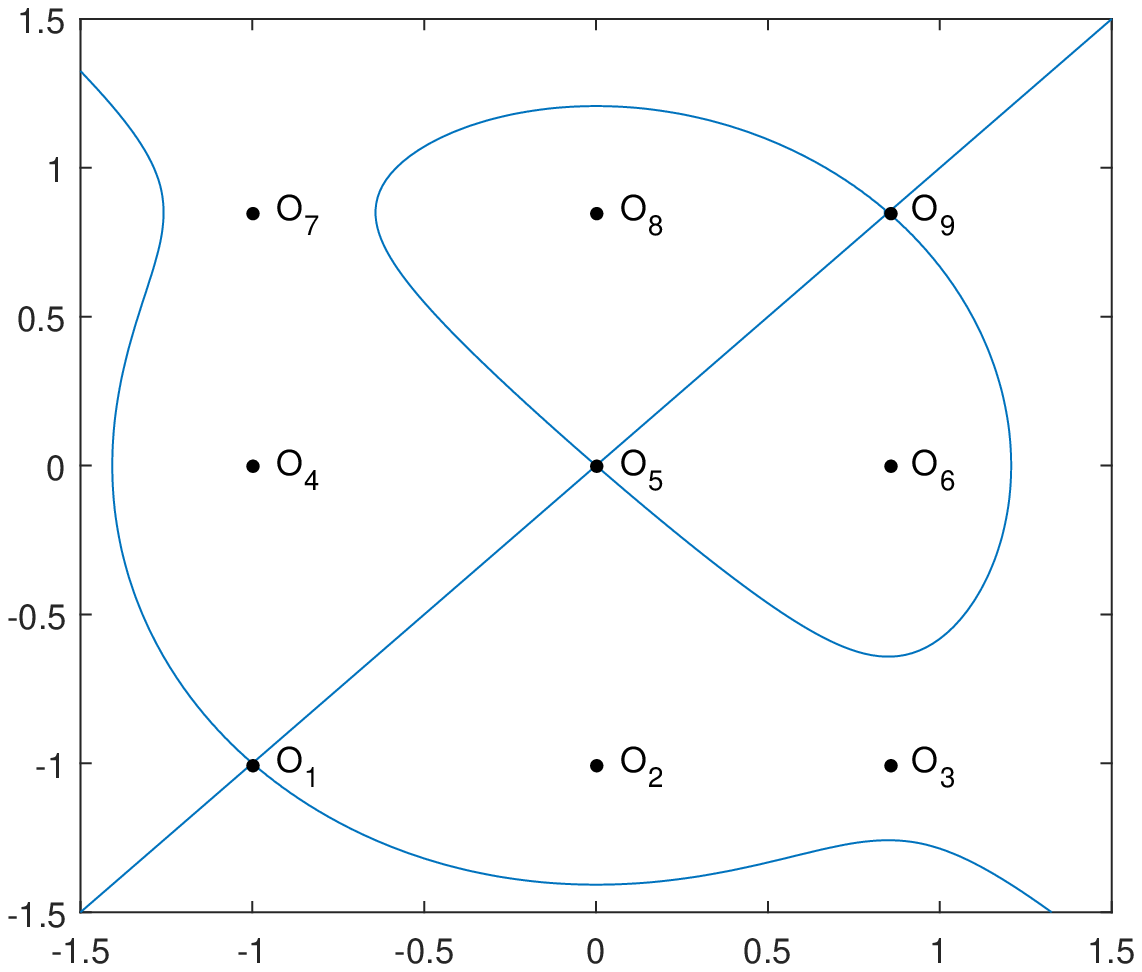} }}%
    \caption{Franz potential $\theta = 0.85$ and equilibrium points of INS $K=2$}%
    \label{fig:b1}%
\end{figure}

To apply the results of \cite{dupwu} we need several quantities that are
constructed in terms of the Freidlin-Wentzell quasipotential. The
quasipotential for \eqref{eqn:dym_INS} is easy to identify because the system
is reversible with $\nu^{\varepsilon}\in\mathcal{P}(M^{K})$ defined by
\eqref{eqn:stat_dis_sym} as its unique stationary distribution. Thus if for
$\boldsymbol{x}\in M^{K}$ we define
\begin{equation}
U(\boldsymbol{x})\doteq \min_{\sigma\in\Sigma_{K}}\left\{  \sum_{\ell=1}^{K}%
\alpha_{\ell}V\left(  x_{\sigma\left(  \ell\right)  }\right)  \right\}  ,
\label{eqn:U}%
\end{equation}
then $U$ corresponds to a potential, and it is easy to see that $U(O_{1})=0$.
Figure \ref{fig:b1} depicts $U$ for the Franz potential.

\begin{figure}[h]
    \centering
    \includegraphics[width=0.9\textwidth]{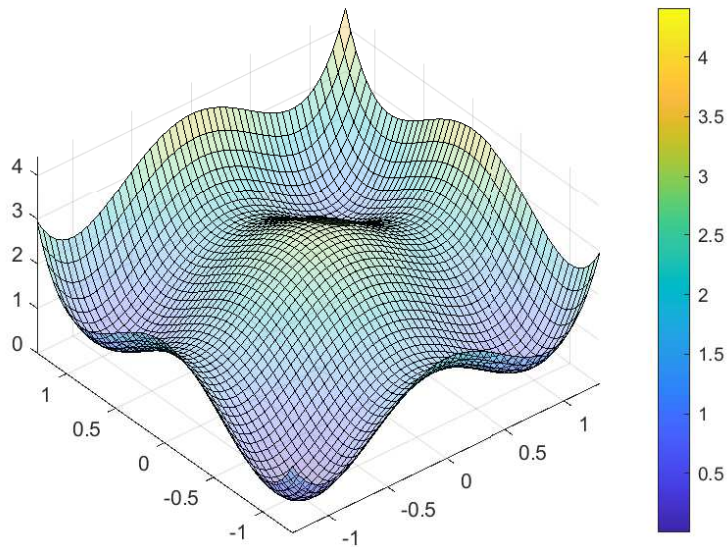}
    \caption{Symmetrized potential for $K=2$}
    \label{fig:b2}
\end{figure}

Since we assume that $\{\boldsymbol{X}^{\varepsilon}(t)\}_{0\leq t\leq T}$ satisfies a large deviation principle on $C([0,T]:M^{K})$ with rate function $I_T:C([0,T]:M^{K})\rightarrow [0,\infty]$ for arbitrary $T\in(0,\infty)$, the
quasipotential $Q(\boldsymbol{x},\boldsymbol{y})$ is defined for all
$\boldsymbol{x},\boldsymbol{y}\in M^{K}$ by
\[
    Q(\boldsymbol{x},\boldsymbol{y})\doteq\inf\left\{  I_{T}(\phi):\phi(0)=\boldsymbol{x},\phi(T)=\boldsymbol{y},T<\infty\right\}  .
\]
(in fact the specific form of the
quasipotential is already known since we know the rate function for the
stationary distributions $\{\nu^{\varepsilon}\}$). 

\vspace{\baselineskip}
Next we give a definition from graph theory which will be used in the proofs of the main results.

\begin{definition}
\label{Def:3.3}Given a subset $W\subset L=\{1,\ldots,l\},$ a directed graph
consisting of arrows $i\rightarrow j$ $(i\in L\setminus W,j\in L,i\neq j)$ is
called a $W$\textbf{-graph on }$L$ if it satisfies the following conditions.

\begin{enumerate}
\item Every point $i$ $\in L\setminus W$ is the initial point of exactly one arrow.



\item For any point $i$ $\in L\setminus W,$ there exists a sequence of
arrows leading from $i$ to some point in $W.$
\end{enumerate}
\end{definition}

We note that we could replace the second condition by the requirement that there are no closed cycles in the graph.
We denote by $G(W)$ the set of $W$-graphs; we shall use the letter $g$ to denote graphs. 

\begin{remark}
\label{Rmk:3.1}We use $G(i)$ to denote $G(\{i\})$ and  $G(i,j)$ to denote $G(\{i,j\}).$
\end{remark}

\begin{definition}
    For all $j\in L$, define
    \begin{equation}
        W\left(  O_{j}\right)  \doteq\min_{g\in G\left(  j\right)  }\left[{\textstyle\sum_{\left(  m\rightarrow n\right)  \in g}}V\left(  O_{m},O_{n}\right)  \right],   \label{eqn:Wtwarg}
    \end{equation}
    \begin{equation}
        W\left( O_1\cup O_{j}\right)  \doteq\min_{g\in G\left(  1,j\right)  }\left[{\textstyle\sum_{\left(  m\rightarrow n\right)  \in g}}V\left(  O_{m},O_{n}\right)  \right]  ,
        \label{eqn:Wtwarg_2}
    \end{equation}
and 
    \begin{equation}
        W(\boldsymbol{x})\doteq\min_{i\in L}\left[  W(O_{i})+Q(O_{i},\boldsymbol{x})\right]  . \label{eqn:W_x}%
    \end{equation}
\end{definition}
\begin{remark}
\label{rmk:roleofW}
Heuristically, if we interpret
$V\left(  O_{m},O_{n}\right)  $ as the \textquotedblleft
cost\textquotedblright\ of moving from $O_{m}$ to $O_{n},$ then $W\left(O_{j}\right)  $ is the ``least total cost''  
of reaching $O_{j}$ from every $O_{i}$ with $i\in L\setminus\{j\}.$
\end{remark}

Before proceeding to the next subsection, we state and prove a lemma that ties
up the relation between $W$ and $U$. The relation will also be used later on
for solving the optimization problem

\begin{lemma}
\label{Lem:4.1} For any $\boldsymbol{x},\boldsymbol{y}\in M^{K},$ $W\left(
\boldsymbol{x}\right)  -W\left(  \boldsymbol{y}\right)  =U\left(
\boldsymbol{x}\right)  -U\left(  \boldsymbol{y}\right)  .$
\end{lemma}

\begin{proof}
Since we know that the stationary distribution $\nu^{\varepsilon}$ of
$\{\boldsymbol{X}^{\varepsilon}(t)\}_{t\geq0}$ is given by
\eqref{eqn:stat_dis_sym}, we can apply \cite[Theorem 4.3, Chapter 6]{frewen2}
to find that for any $\eta>0$ and for sufficiently small neighborhoods of
$\boldsymbol{x}$ and $\boldsymbol{y}$,%
\[
\frac{\nu^{\varepsilon}\left(  B_{\delta}\left(  \boldsymbol{x}\right)
\right)  }{\nu^{\varepsilon}\left(  B_{\delta}\left(  \boldsymbol{y}\right)
\right)  }\leq\frac{\exp\left\{  -\frac{1}{\varepsilon}\left(  W\left(
\boldsymbol{x}\right)  -\min_{i\in L}W(O_{i})-\eta\right)  \right\}  }%
{\exp\left\{  -\frac{1}{\varepsilon}\left(  W\left(  \boldsymbol{y}\right)
-\min_{i\in L}W(O_{i})+\eta\right)  \right\}  }=e^{-\frac{1}{\varepsilon
}\left(  W\left(  \boldsymbol{x}\right)  -W\left(  \boldsymbol{y}\right)
-2\eta\right)  }%
\]
and
\[
\frac{\nu^{\varepsilon}\left(  B_{\delta}\left(  \boldsymbol{x}\right)
\right)  }{\nu^{\varepsilon}\left(  B_{\delta}\left(  \boldsymbol{y}\right)
\right)  }\geq\frac{\exp\left\{  -\frac{1}{\varepsilon}\left(  W\left(
\boldsymbol{x}\right)  -\min_{i\in L}W(O_{i})+\eta\right)  \right\}  }%
{\exp\left\{  -\frac{1}{\varepsilon}\left(  W\left(  \boldsymbol{y}\right)
-\min_{i\in L}W(O_{i})-\eta\right)  \right\}  }=e^{-\frac{1}{\varepsilon
}\left(  W\left(  \boldsymbol{x}\right)  -W\left(  \boldsymbol{y}\right)
+2\eta\right)  }.
\]
Thus
\[
\limsup_{\varepsilon\rightarrow0}-\varepsilon\log\left(  \frac{\nu
^{\varepsilon}\left(  B_{\delta}\left(  \boldsymbol{x}\right)  \right)  }%
{\nu^{\varepsilon}\left(  B_{\delta}\left(  \boldsymbol{y}\right)  \right)
}\right)  \leq W\left(  \boldsymbol{x}\right)  -W\left(  \boldsymbol{y}%
\right)  +2\eta
\]
and
\[
\liminf_{\varepsilon\rightarrow0}-\varepsilon\log\left(  \frac{\nu
^{\varepsilon}\left(  B_{\delta}\left(  \boldsymbol{x}\right)  \right)  }%
{\nu^{\varepsilon}\left(  B_{\delta}\left(  \boldsymbol{y}\right)  \right)
}\right)  \geq W\left(  \boldsymbol{x}\right)  -W\left(  \boldsymbol{y}%
\right)  -2\eta.
\]
On the other hand, for $\boldsymbol{w}=\boldsymbol{x},\boldsymbol{y}$ the
definition of $U$ implies%
\begin{align*}
\int_{B_{\delta}\left(  \boldsymbol{w}\right)  }\exp\left\{  -\frac
{1}{\varepsilon}U\left(  \boldsymbol{z}\right)  \right\}  d\boldsymbol{z}  &
\leq\int_{B_{\delta}\left(  \boldsymbol{x}\right)  }\left[  \sum_{\sigma
\in\Sigma_{K}}\exp\left\{  -\frac{1}{\varepsilon}\sum_{\ell=1}^{K}\alpha
_{\ell}V\left(  z_{\sigma\left(  \ell\right)  }\right)  \right\}  \right]
d\boldsymbol{z}\\
&  \leq K!\cdot\int_{B_{\delta}\left(  \boldsymbol{w}\right)  }\exp\left\{
-\frac{1}{\varepsilon}U\left(  \boldsymbol{z}\right)  \right\}
d\boldsymbol{z}.
\end{align*}
Therefore
\begin{align*}
&  \lim_{\varepsilon\rightarrow0}-\varepsilon\log\left(  \frac{\nu
^{\varepsilon}\left(  B_{\delta}\left(  \boldsymbol{x}\right)  \right)  }%
{\nu^{\varepsilon}\left(  B_{\delta}\left(  \boldsymbol{y}\right)  \right)
}\right) \\
&  \quad=\lim_{\varepsilon\rightarrow0}-\varepsilon\log\left(  \frac
{\int_{B_{\delta}\left(  \boldsymbol{x}\right)  }\left[  \sum_{\sigma\in
\Sigma_{K}}\exp\left\{  -\frac{1}{\varepsilon}\sum_{\ell=1}^{K}\alpha_{\ell
}V\left(  z_{\sigma\left(  \ell\right)  }\right)  \right\}  \right]
d\boldsymbol{z}}{\int_{B_{\delta}\left(  \boldsymbol{y}\right)  }\left[
\sum_{\sigma\in\Sigma_{K}}\exp\left\{  -\frac{1}{\varepsilon}\sum_{\ell=1}%
^{K}\alpha_{\ell}V\left(  z_{\sigma\left(  \ell\right)  }\right)  \right\}
\right]  d\boldsymbol{z}}\right) \\
&  \quad=\lim_{\varepsilon\rightarrow0}-\varepsilon\log\left(  \frac
{\int_{B_{\delta}\left(  \boldsymbol{x}\right)  }\exp\left\{  -\frac
{1}{\varepsilon}U\left(  \boldsymbol{z}\right)  \right\}  d\boldsymbol{z}%
}{\int_{B_{\delta}\left(  \boldsymbol{y}\right)  }\exp\left\{  -\frac
{1}{\varepsilon}U\left(  \boldsymbol{z}\right)  \right\}  d\boldsymbol{z}%
}\right) \\
&  \quad=\min_{\boldsymbol{u}\in B_{\delta}\left(  \boldsymbol{x}\right)
}U\left(  \boldsymbol{u}\right)  -\min_{\boldsymbol{u}\in B_{\delta}\left(
\boldsymbol{y}\right)  }U\left(  \boldsymbol{u}\right)  ,
\end{align*}
where the last equality is from Laplace's principle. Hence $\min
_{\boldsymbol{u}\in B_{\delta}\left(  \boldsymbol{x}\right)  }U\left(
\boldsymbol{u}\right)  -\min_{\boldsymbol{u}\in B_{\delta}\left(
\boldsymbol{y}\right)  }U\left(  \boldsymbol{u}\right)  $ is between $W\left(
\boldsymbol{x}\right)  -W\left(  \boldsymbol{y}\right)  \pm2\eta.$ Sending
$\eta\rightarrow0$ (and thus $\delta\rightarrow0$), we find $W\left(
\boldsymbol{x}\right)  -W\left(  \boldsymbol{y}\right)  =U\left(
\boldsymbol{x}\right)  -U\left(  \boldsymbol{y}\right)  .$
\end{proof}

\begin{remark}
\label{rmk:4.1} By \eqref{eqn:W_x} and Lemma \ref{Lem:4.1},
\[
U(\boldsymbol{x})=\min_{i\in L}\left[  U(O_{i})+Q(O_{i},\boldsymbol{x}%
)\right]  .
\]

\end{remark}

We can now state the main result of \cite{dupwu}. The result stated in
\cite{dupwu} assumes a fixed function $f$, but the result as stated below
follows from this and the uniform convergence $f_{\varepsilon}\rightarrow f$.
The uniformity with respect to the initial condition is discussed on \cite[page 12]{dupwu}.
Let
\begin{equation}
h\doteq\min_{i\in L\setminus\{1\}}Q(O_{1},O_{i})\text{ and }w\doteq
W(O_{1})-\min_{i\in L\setminus\{1\}}W(O_{1}\cup O_{i}).
 \label{eqn:somedefs}
\end{equation}
The quantity $h$ is related to the time that it takes for the process to leave a neighborhood of $O_1$, and $W(O_1)-W(O_1\cup O_i)$ is related to the transition time from a neighborhood of $O_i$ to one of $O_1$. The roles of $h$ and $w$ will be further explained in Section \ref{sec:6}.  

\begin{theorem}
\label{thm:lower_bound}
Assume that the process defined by (\ref{eqn:dym_INS}) satisfies a large deviation principle that is uniform with respect to initial conditions, and let $\nu^{\varepsilon}$ be its unique stationary distribution and let $T^{\varepsilon}=e^{\frac{1}{\varepsilon}c}$ for some $c>h\vee w$.
Suppose that for each $\varepsilon>0$ $f_{\varepsilon}:M^{K}\rightarrow
\mathbb{R}$, and that for a continuous function $f:M^{K}\rightarrow\mathbb{R}$
we have $f_{\varepsilon}\rightarrow f$ uniformly on $M^{K}$. Then for any compact set $A\subset M^{K}$ and $\boldsymbol{x} \in M^{K}$, 
\begin{align*}
&  \liminf_{\varepsilon\rightarrow0}-\varepsilon\log\left\vert E_{\boldsymbol{x}}\left(  \frac{1}{T^{\varepsilon}}\int_{0}^{T^{\varepsilon}%
}e^{-\frac{1}{\varepsilon}f_{\varepsilon}\left(  X_{t}^{\varepsilon}\right)  }1_{A}\left(
X_{t}^{\varepsilon}\right)  dt\right)  -\int_{M^K} e^{-\frac{1}{\varepsilon}f_{\varepsilon}\left(
\boldsymbol{x}\right)  }1_{A}\left(  \boldsymbol{x}\right)  \nu^{\varepsilon}\left(  d\boldsymbol{x}\right)
\right\vert \\
&  \qquad\geq\inf_{\boldsymbol{x}\in A}\left[  f\left(  \boldsymbol{x}\right)  +W\left(  \boldsymbol{x}\right)
\right]  -W\left(  O_{1}\right)  +c-(h\vee w),
\end{align*}
and 
\begin{align*}
&  \liminf_{\varepsilon\rightarrow0}-\varepsilon\log\left(  T^{\varepsilon
}\cdot\text{$\mathrm{Var}_{\boldsymbol{x}}\left(  \frac{1}{T^{\varepsilon}%
}\int_{0}^{T^{\varepsilon}}e^{-\frac{1}{\varepsilon}f_{\varepsilon}%
(X_{t}^{\varepsilon})}1_{A}(X_{t}^{\varepsilon})dt\right)  $}\right) \\
&  \quad\geq%
\begin{cases}
\min_{i\in L}\left(  R_{i}^{(1)}\wedge R_{i}^{(2)}\right)  , & \text{if
}h\geq w\\
\min_{i\in L}\left(  R_{i}^{(1)}\wedge R_{i}^{(2)}\wedge R_{i}^{(3)}\right), & \text{otherwise }%
\end{cases}
,
\end{align*}
where
\[
R_{i}^{(1)}\doteq\inf_{\boldsymbol{x}\in A}[2f(\boldsymbol{x})+Q(O_{i}%
,\boldsymbol{x})]+W(O_{i})-W(O_{1}),
\]%
\[
R_{1}^{(2)}\doteq2\inf_{\boldsymbol{x}\in A}[f(\boldsymbol{x})+Q(O_{1}%
,\boldsymbol{x})]-h,
\]
for $i\in L\setminus\{1\}$
\[
R_{i}^{(2)}\doteq2\inf_{\boldsymbol{x}\in A}[f(\boldsymbol{x})+Q(O_{i}%
,\boldsymbol{x})]+W(O_{i})-2W(O_{1})+W(O_{1}\cup O_{i}),
\]
and for $i \in L$
\[
R_{i}^{(3)}\doteq2\inf_{\boldsymbol{x}\in A}[f(\boldsymbol{x})+Q(O_{i}%
,\boldsymbol{x})]+2W(O_{i})-2W(O_{1})-w.
\]

\end{theorem}

To apply this theorem to the INS model, we note that the definition of
$\theta_{\text{INS}}^{\varepsilon,T^{\varepsilon}}$ involves the sum of a
finite number of integrals of the form
\[
\frac{1}{T^{\varepsilon}}\int_{0}^{T^{\varepsilon}}w^{\varepsilon
}(\boldsymbol{X}_{\sigma}^{\varepsilon}(t),\boldsymbol{\alpha})1_{A}(X_{\sigma
(1)}^{\varepsilon}(t))dt,
\]
where $\sigma$ is a permutation which for simplicity we take here to be the
identity, and $w^{\varepsilon}(\boldsymbol{x},\boldsymbol{\alpha})$ is defined
in (\ref{eqn:defofw}). Since $V$ is bounded and continuous, it follows from
standard features of the mollification used in the definition of
$w^{\varepsilon}$ in (\ref{eqn:defofw}), that if we write $w^{\varepsilon
}(\boldsymbol{x},\boldsymbol{\alpha})$ in the form
\[
e^{-\frac{1}{\varepsilon}\sum_{\ell=1}^{K}\alpha_{\ell}V(x_{\ell})+\frac
{1}{\varepsilon}g_{\varepsilon}(\boldsymbol{x},\boldsymbol{\alpha})},
\]
then as $\varepsilon\rightarrow0$%
\begin{equation}
g_{\varepsilon}(\boldsymbol{x},\boldsymbol{\alpha})\rightarrow U(\boldsymbol{x})\doteq \min_{\sigma
\in\Sigma_{K}}\left[  \sum_{\ell=1}^{K}\alpha_{\ell}V(x_{\sigma(\ell)})\right]
 \label{eqn:convoff}%
\end{equation}
uniformly in $\boldsymbol{x}\in M^{K}$ (see, e.g., \cite[Lemma 14.7]%
{buddup4}). Define
\[
f(\boldsymbol{x},\boldsymbol{\alpha})
=\sum_{\ell=1}^{K}\alpha_{\ell}V(x_{\ell})-U(\boldsymbol{x}).
\]
We can then apply Theorem \ref{thm:lower_bound} with the function
$f_{\varepsilon}(\boldsymbol{x},\boldsymbol{\alpha})=\sum_{\ell=1}^{K}%
\alpha_{\ell}V(x_{\ell})-g_{\varepsilon}(\boldsymbol{x},\boldsymbol{\alpha})$
and the compact set $A\times M^{K-1}\subset M^{K}$, to find that
\begin{align*}
&  \liminf_{\varepsilon\rightarrow0}-\varepsilon\log\left\vert E_{\boldsymbol{x}}\left(  \theta_{\text{INS}}^{\varepsilon,T^{\varepsilon}}\right)  - \nu^{\varepsilon}(A)\right\vert \\
&  \qquad\geq\inf_{\boldsymbol{x}\in A\times M^{K-1}}\left[  f\left(  \boldsymbol{x},\boldsymbol{\alpha}\right)  +W\left(  \boldsymbol{x}\right)
\right]  -W\left(  O_{1}\right)  +c-(h\vee w)\\
& \qquad= \inf_{\boldsymbol{x}\in A\times M^{K-1}}\left[  f\left(  \boldsymbol{x},\boldsymbol{\alpha}\right)  +U\left(  \boldsymbol{x}\right)
\right] +c-(h\vee w).
\end{align*}
Since $f\geq 0$, $U\geq 0$ and $c>h\vee w$, this shows that $\theta_{\text{INS}}^{\varepsilon,T^{\varepsilon}}$ is essentially unbiased. Moreover, we find that
$\liminf_{\varepsilon\rightarrow0}-\varepsilon\log(T^{\varepsilon}%
\cdot\text{$\mathrm{Var}_{x}(\theta_{\text{INS}}^{\varepsilon,T^{\varepsilon}%
})$})$ is bounded below by either $\min_{i\in L}(R_{i}^{(1)}%
(\boldsymbol{\alpha})\wedge R_{i}^{(2)}(\boldsymbol{\alpha}))$ or $\min_{i\in
L}(R_{i}^{(1)}(\boldsymbol{\alpha})\wedge R_{i}^{(2)}(\boldsymbol{\alpha
})\wedge R_{i}^{(3)}(\boldsymbol{\alpha}))$, depending on whether $h\geq
w$ or $w>h$. 

In the next subsection, we will identify appropriate lower bounds for these two minima and then optimize the lower bounds over $\boldsymbol{\alpha}$.

\begin{remark}
    As mentioned in Remark \ref{rmk:quants}, we are also interested in estimating risk sensitive functionals of the form 
    \[
        \int_{\mathbb{R}^{d}}e^{-\frac{1}{\varepsilon}F\left(  x\right)  }\mu^{\varepsilon}\left(  dx\right).  
    \]
    We can apply Theorem \ref{thm:lower_bound} to the associated INS estimator in this case as well by using the function 
    $f_{\varepsilon}(\boldsymbol{x},\boldsymbol{\alpha})=F(x_1)+\sum_{\ell=1}^{K}\alpha_{\ell}V(x_{\ell})-g_{\varepsilon}(\boldsymbol{x},\boldsymbol{\alpha})$ and the compact set $M^{K}$. 
    Moreover, one can modify the arguments in Subsection \ref{subsec_5.1} to derive an analogous version of Theorem \ref{thm:lower_bound_INS} for the risk sensitive functional case.
\end{remark}

\subsection{Bounds for the optimization problem}
\label{subsec_5.1}
In this subsection we provide suitable lower bounds for $\min_{i\in L}%
(R_{i}^{(1)}(\boldsymbol{\alpha})\wedge R_{i}^{(2)}(\boldsymbol{\alpha}))$ and
$\min_{i\in L}(R_{i}^{(1)}(\boldsymbol{\alpha})\wedge R_{i}^{(2)}%
(\boldsymbol{\alpha})\wedge R_{i}^{(3)}(\boldsymbol{\alpha}))$.
Define
\[
r(\boldsymbol{\alpha})\doteq\inf_{\boldsymbol{x}\in A\times M^{K-1}}\left\{
2\sum_{\ell=1}^{K}\alpha_{\ell}V(x_{\ell})-\min_{\sigma\in\Sigma_{K}}\left\{
\sum_{\ell=1}^{K}\alpha_{\ell}V(x_{\sigma(\ell)})\right\}  \right\}  ,
\]
which is the same as $\inf_{\boldsymbol{x}\in A\times M^{K-1}}\left\{
2f(\boldsymbol{x},\boldsymbol{\alpha})+U(\boldsymbol{x})\right\}  $, where
$f(\boldsymbol{x},\boldsymbol{\alpha})\doteq\sum_{\ell=1}^{K}\alpha_{\ell}V(x_{\ell
})-U(\boldsymbol{x})$. As the next lemma shows, this optimization problem,
which plays a key role in the bounds we will derive, has an explicit solution.
Although a proof appears in \cite{dupwusna}, we include a slightly simpler proof of the special case needed here owing to its central role.

\begin{lemma}
We have
\[
\sup_{\boldsymbol{\alpha}\in\Delta}r(\boldsymbol{\alpha})=\left(  2-\left(
1/2\right)  ^{K-1}\right)  V(A),
\]
with the unique optimizer $\boldsymbol{\alpha}^{\ast}=(1,1/2,\dots
,(1/2)^{K-1}).$
\end{lemma}
\begin{proof}
The first step is to decompose 
$A\times M^{K-1}$ as $\cup_{\tau\in\Sigma_{K}}N_{\tau},$ where
\[
    N_{\tau}\doteq\left\{  \boldsymbol{x}\in A\times M^{K-1}:V\left(  x_{\tau\left(1\right)  }\right)  \leq V\left(  x_{\tau\left(  2\right)  }\right)\leq\cdots\leq V\left(  x_{\tau\left(  K\right)  }\right)  \right\}  .
\]
For any $\tau\in\Sigma_{K}$ there exists $i\in\{1,\ldots,K\}$ which depends on
$\tau$ such that $1=\tau\left( i\right)  .$ We will use the rearrangement inequality \cite[Section 10.2, Theorem 368]{harlitpol}, which says that if
$\boldsymbol{x}\in N_{\tau},$ then since $\alpha_{\ell}$ is nonincreasing in $\ell$ the
minimum in $U(\boldsymbol{x}) \doteq \min_{\sigma\in\Sigma_{K}}\{  \sum_{\ell=1}^{K}\alpha_{\ell}V\left(x_{\sigma\left(  \ell\right)  }\right)  \}$
is at $\sigma=\tau$. Thus,
\begin{align*}
    &\inf_{\boldsymbol{x}\in A\times M^{K-1}}\left[  2\sum_{\ell=1}^{K}%
    \alpha_{\ell}V\left(  x_{_{\ell}}\right)   -U(\boldsymbol{x})  \right] \\
    &\quad  =\min_{\tau\in\Sigma_{K}}\left\{  \inf_{\boldsymbol{x}\in N_{\tau}}\left[2\sum_{\ell=1}^{K}\alpha_{\ell}V\left(  x_{_{\ell}}\right)  -\min_{\sigma\in\Sigma_{K}}\left\{  \sum_{\ell=1}^{K}\alpha_{\ell}V\left(x_{\sigma\left(  \ell\right)  }\right)  \right\}  \right]  \right\} \\
    &\quad  =\min_{\tau\in\Sigma_{K}}\left\{  \inf_{\boldsymbol{x}\in N_{\tau}}\left[\sum_{\ell=1}^{K}\left(  2\alpha_{\tau\left(  \ell\right)  }-\alpha_{\ell}\right)V\left(  x_{\tau\left(  \ell\right)  }\right)   \right]
\right\}  .
\end{align*}

Let $\beta_{\ell}\doteq2\alpha_{\tau\left(  \ell\right)  }-\alpha_{\ell}$, and for each
$i\in\{1,\ldots,K\}$ define the sets
\[
N_{\tau}^{i}\doteq\left\{  \left(  x_{\tau\left(  1\right)  },\ldots
,x_{\tau\left(  i \right)  }\right)  :\boldsymbol{x}\in N_{\tau}\right\}
\]
and
\[
\bar{N}_{\tau}^{i}\left(  \boldsymbol{y}\right)  \doteq\left\{  \left(
x_{\tau\left(  i\right)  },\ldots,x_{\tau\left(  K\right)  }\right)
:\boldsymbol{x}\in N_{\tau}\text{ and }\left(  x_{\tau\left(  1\right)  }%
,\ldots,x_{\tau\left(  i\right)  }\right)  =\boldsymbol{y}\right\}  .
\]
Note that for each $\tau$ (and using that $i$ is the index such that
$\tau\left(  i\right)  =1$)%
\begin{align*}
     &\inf_{\boldsymbol{x}\in N_{\tau}}\left[  \sum_{\ell=1}^{K}\beta_{\ell}V\left(  x_{\tau\left(  \ell\right)  }\right)  \right] \\
    &  =\inf_{\left(  y_{1},\ldots,y_{i}\right)  \in N_{\tau}^{i}}\left[
        \begin{array}[c]{c}%
            \sum_{\ell=1}^{i-1}\beta_{\ell}V\left(  y_{\ell}\right)  +\beta_{i}V\left(y_{i}\right)   \\
            +\inf_{\left(  z_{i},\ldots,z_{K}\right)  \in\bar{N}_{\tau}^{i}\left(y_{1},\ldots,y_{i}\right)  }\left[  \sum_{\ell=i+1}^{K}\beta_{\ell}V\left(z_{\ell}\right)  \right]
        \end{array}
    \right]  .
\end{align*}

Next we show that given $\left(  y_{1},\ldots,y_{i}\right)  $ (and noting that
by definition $z_{i}=y_{i}$),
\begin{equation}
\inf_{\left(  z_{i},\ldots,z_{K}\right)  \in\bar{N}_{\tau}^{i}\left(
y_{1},\ldots,y_{i}\right)  }\left[  \sum_{\ell=i+1}^{K}\beta_{\ell}V\left(
z_{\ell}\right)  \right]  =\left(  \sum_{\ell=i+1}^{K}\beta_{\ell}\right)  V\left(
y_{i}\right)  . \label{eqn:parconst}%
\end{equation}
Recall that $\alpha_1\geq\alpha_2\cdots\geq\alpha_K>0$. Therefore, $\beta_K=2\alpha_{\tau(K)}-\alpha_K\geq 2\alpha_K-\alpha_K=\alpha_K>0$. More generally, since $\tau(\ell),\dots,\tau(K)$ are distinct values drawn from $\{1,\dots,K\}$, for each $\ell$
\[
    \beta_{\ell}+\dots+\beta_K = 2\sum_{j=\ell}^K \alpha_{\tau(j)} - \sum_{j=\ell}^K \alpha_j \geq 2 \sum_{j=\ell}^K \alpha_j- \sum_{j=\ell}^K \alpha_j>0.
\]
Using
$\beta_{K}\geq0$ and the fact that $\left(  z_{i},\ldots,z_{K}\right)  \in
\bar{N}_{\tau}^{i}\left(  y_{1},\ldots,y_{i}\right)  $ implies the restriction%
\[
    V\left(  z_{i}\right)  \leq V\left(  z_{i+1}\right)  \leq\cdots\leq V\left(z_{K}\right)  ,
\]
we can rewrite the infimum as
\begin{align*}
&  \inf_{\left(  z_{i},\ldots,z_{K}\right)  \in\bar{N}_{\tau}^{i}\left(
y_{1},\ldots,y_{i}\right)  }\left[ \sum_{\ell=i+1}^{K}\beta_{\ell}V\left(z_{\ell}\right)  \right] \\
&  \quad=\inf_{\left(  z_{i},\ldots,z_{K}\right)  \in\bar{N}_{\tau}^{i}\left(y_{1},\ldots,y_{i}\right)  }\left[  \sum_{\ell=i+1}^{K-2}\beta_{\ell}V\left(z_{\ell}\right)  +\left(  \beta_{K-1}+\beta_{K}\right)  V\left(  z_{K-1}\right)
\right]  .
\end{align*}
Iterating, we have (\ref{eqn:parconst}). 
Letting $D\doteq\{V(x):x\in A\}$,%
\begin{align*}
&  \inf_{\boldsymbol{x}\in A\times M^{K-1}}\left[  2\sum_{\ell=1}^{K}\alpha
_{\ell}V\left(  x_{_{\ell}}\right)    -\min_{\sigma\in
\Sigma_{K}}\left\{  \sum_{\ell=1}^{K}\alpha_{\ell}V\left(  x_{\sigma\left(
\ell\right)  }\right)  \right\}  \right] \\
&  =\min_{\tau\in\Sigma_{K}}\left\{  \inf_{\boldsymbol{x}\in N_{\tau}}\left[
\sum_{\ell=1}^{K}\left(  2\alpha_{\tau\left(  \ell\right)  }-\alpha_{\ell}\right)
V\left(  x_{\tau\left(  \ell\right)  }\right)   \right]
\right\} \\
&  =\min_{\tau\in\Sigma_{K}}\left\{  \inf_{\left(  x_{\tau\left(  1\right)
},\ldots,x_{\tau\left(  i\right)  }\right)  \in N_{\tau}^{i}}\left[
\sum_{\ell=1}^{i-1}\beta_{\ell}V\left(  x_{\tau\left(  \ell\right)  }\right)  +\left(
\sum_{\ell=i}^{K}\beta_{\ell}\right)  V\left(  x_{\tau\left(  i\right)  }\right)
 \right]  \right\} \\
&  =\min_{\tau\in\Sigma_{K}}\left\{  \inf_{\substack{\left\{  
V_{\tau\left(  i\right)}    \in D\right\}  \\\left\{  \left(  V_{\tau\left(  1\right)
},\ldots,V_{\tau\left(  i-1\right)  }\right)  :V_{\tau\left(  1\right)  }\leq
V_{\tau\left(  2\right)  }\leq\cdots\leq V_{\tau\left(  i\right)  }\right\}
}}\left[  \sum_{\ell=1}^{i-1}\beta_{\ell}V_{\tau\left(  \ell\right)  }+\left(
\sum_{\ell=i}^{K}\beta_{\ell}\right)  V_{\tau\left(  i\right)  }\right]
\right\}  .
\end{align*}
The last equality holds because $V$ is continuous.

We claim that the last display coincides with 
\begin{align*}
\bar{r}\left(  \boldsymbol{\alpha}\right)   &  \doteq\inf_{\substack{
\{ V_{1}  \in D\}\\\left\{  (V_{1},\ldots,V_{K}):V_{\ell}\in\lbrack
0,V_{1}]\text{ for }\ell\geq2\right\}  }}\left[  2\sum_{\ell=1}^{K}\alpha_{\ell}%
V_{\ell}-\min_{\sigma\in\Sigma_{K}}\left\{  \sum_{\ell=1}^{K}\alpha_{\textbf{}}V_{\sigma\left(  \ell\right)  }\right\}  \right] \\
&  =\min_{\tau\in\Sigma_{K}}\left\{  \inf_{_{\substack{\left\{ 
V_{\tau\left(  i\right)}  \in D\right\}  \\\left\{  \left(  V_{\tau\left(  1\right)
},\cdots,V_{\tau\left(  K\right)  }\right)  :V_{\tau\left(  1\right)  }\leq
V_{\tau\left(  2\right)  }\leq\cdots\leq V_{\tau\left(  K\right)  }\leq
V_{\tau\left(  i\right)  }\right\}  }}}\left[  \sum_{\ell=1}^{K}\left(
2\alpha_{\tau\left(  \ell\right)  }-\alpha_{\ell}\right)  V_{\tau\left(  \ell\right)
}\right]  \right\}  .
\end{align*}
Since $\boldsymbol{V}\in N_{\tau}$ implies $V_{\tau\left(  \ell\right)  }\geq
V_{\tau\left(  i\right)  }$ and hence $V_{\tau\left(  \ell\right)  }%
=V_{\tau\left(  i\right)  }$ for $i<\ell\leq K$,%
\begin{align*}
&  \inf_{_{_{\substack{\left\{   V_{\tau\left(  i\right)  } \in D\right\}  \\\left\{
\left(  V_{\tau\left(  1\right)  },\ldots,I_{\tau\left(  K\right)  }\right)
:V_{\tau\left(  1\right)  }\leq V_{\tau\left(  2\right)  }\leq\cdots\leq
V_{\tau\left(  K\right)  }\leq V_{\tau\left(  i\right)  }\right\}  }}}}\left[
\sum_{\ell=1}^{K}\beta_{\ell}V_{\tau\left(  \ell\right)  }\right] \\
&  \quad=\inf_{_{\substack{\left\{   V_{\tau\left(  i\right)
} \in D\right\}
\\\left\{ \left( V_{\tau\left(  1\right)  },\ldots,V_{\tau\left(
i-1\right)  }\right)  :V_{\tau\left(  1\right)  }\leq V_{\tau\left(  2\right)
}\leq\cdots\leq V_{\tau\left(  i\right)  }\right\}  }}}\left[  \sum
_{\ell=1}^{i-1}\beta_{\ell}V_{\tau\left(  \ell\right)  }+\left(  \sum_{\ell=i}^{K}%
\beta_{\ell}\right)  V_{\tau\left(  i\right)  }\right]  ,
\end{align*}
which establishes the claim.

To prove that $\sup_{\boldsymbol{\alpha}}\bar{r}\left(
\boldsymbol{\alpha}\right)  =\{(2-\left(  1/2\right)  ^{K-1})V\left(
A\right)   \},$ first rewrite $\bar{r}\left(
\boldsymbol{\alpha}\right)  $ by noticing that since $V_{1}$ is the largest
value in the set $\boldsymbol{V}$,%
\[
\min_{\tau\in\Sigma_{K}}\left\{  \sum_{\ell=1}^{K}\alpha_{\ell}V_{\tau\left(
\ell\right)  }\right\}
\]
obtains the minimum at some $\tau\in\Sigma_{K}$ with $\tau\left(  K\right)
=1$. Therefore
\begin{align*}
&  \bar{r}\left(  \boldsymbol{\alpha}\right) \\
&  =\inf_{\substack{\left(  V_{1},F\right)  \in D\\\left\{  \boldsymbol{V}%
:V_{\ell}\leq V_{1}\text{ for }\ell\geq2\right\}  }}\left[  \left(  2\alpha
_{1}-\alpha_{K}\right)  V_{1}+2\sum_{\ell=2}^{K}\alpha_{\ell}V_{\ell}-\min_{\tau
\in\Sigma_{K},\tau\left(  K\right)  =1}\left\{  \sum_{\ell=1}^{K-1}\alpha
_{\ell}V_{\tau\left(  \ell\right)  }\right\}  \right]  .
\end{align*}
Suppose we are given any $K-1$ numbers and assign them to $\left\{
V_{\ell}\right\}_{\ell=2,\ldots,K}$ in a certain order. Then the value of%
\[
\min_{\tau\in\Sigma_{K},\tau\left(  K\right)  =1}\left\{  \sum_{\ell=1}%
^{K-1}\alpha_{\ell}V_{\tau\left(  \ell\right)  }\right\}
\]
is independent of the order. But since $\alpha_{1}\geq\cdots\geq\alpha_{K}%
\geq0$, by the rearrangement inequality, the smallest value of 
$
\sum_{\ell=2}^{K}\alpha_{\ell}V_{\ell}%
$ 
is obtained by taking the $V_{\ell},$ $\ell\geq2$ in increasing order. By choosing this ordering of $\left\{  V_{\ell}\right\}_{\ell=2,\ldots,K}$,
\[
\min_{\tau\in\Sigma_{K},\tau\left(  K\right)  =1}\left\{  \sum_{\ell=1}%
^{K-1}\alpha_{\ell}V_{\tau\left(  \ell\right)  }\right\}  =\sum_{\ell=2}^{K}\alpha_{\ell-1}V_{\ell}.
\]
Thus,%
\begin{align}
 \bar{r}\left(  \boldsymbol{\alpha}\right) &  =\inf_{\substack{  V_{1}  \in D\\\left\{  \boldsymbol{V}:0\leq V_{2}\leq\cdots\leq V_{K}\leq V_{1}\right\}  }}\left[  \left(  2\alpha_{1}-\alpha_{K}\right)  V_{1}+2\sum_{\ell=2}^{K}\alpha_{\ell}V_{\ell}-\sum_{\ell=2}^{K}\alpha_{\ell-1}V_{\ell}\right] \nonumber\\
&  =\inf_{\substack{ V_{1} \in D\\\left\{  \boldsymbol{V}:0\leq V_{2}\leq\cdots\leq V_{K}\leq V_{1}\right\}  }}\left[  \left(  2\alpha_{1}-\alpha_{K}\right)  V_{1}+\sum_{\ell=2}^{K}\left(  2\alpha_{\ell}-\alpha_{\ell-1}\right)  V_{\ell}\right]  .\label{eqn:Vaexp}
\end{align}
Using summation by parts and $\alpha_{1}=1$, we have
\begin{align}
&  \bar{r}\left(  \boldsymbol{\alpha}\right) \label{eqn:Vaexp2}\\
&  =\inf_{\substack{  V_{1}  \in D\\\left\{  \boldsymbol{V}:0\leq V_{2}\leq\cdots\leq V_{K}\leq V_{1}\right\}  }}\left[  \left(  2\alpha_{1}-\alpha_{K}\right)  V_{1}+\sum_{\ell=2}^{K-1}\alpha_{\ell}\left(  2V_{\ell}-V_{\ell+1}\right)  +2\alpha_{K}V_{K}-V_{2}\right]  .\nonumber
\end{align}
Since $V$ is continuous and bounded from below, there is $ V_{0}  \in \bar{D}$ such that
\[
\left(  2-\left(  1/2\right)  ^{K-1}\right)  V_{0}=\left[  \left(  2-\left(  1/2\right)  ^{K-1}\right)  V\left(A\right)    \right]  .
\]
Let $\boldsymbol{\alpha}^{\ast}\doteq\left(  1,1/2,\ldots,1/2^{K-1}\right)  $ and $\boldsymbol{V}^{\ast}=\left(  V_{1}^{\ast},\ldots,V_{K}^{\ast}\right),$ with $V_{1}^{\ast}\doteq V_{0},$ $V_{\ell}^{\ast}\doteq\left(  1/2\right)^{K-\ell+1}V_{0}$ for $\ell=2,\ldots,K$. We have the following inequalities, which
are explained after the display:
\begin{align*}
&  \left(  2-\left(  1/2\right)  ^{K-1}\right)  V_{0}\\
&\quad  =\inf_{\substack{  V_{1}  \in D\\\left\{  \boldsymbol{V}:0\leq V_{2}\leq\cdots\leq V_{K}\leq V_{1}\right\}  }}\left[  \left(  2\alpha_{1}^{\ast}-\alpha_{K}^{\ast}\right)  V_{1}+\sum_{\ell=2}^{K}\left(  2\alpha_{\ell}^{\ast}-\alpha_{\ell-1}^{\ast}\right)V_{\ell}\right] \\
&\quad  =\bar{r}\left(  \boldsymbol{\alpha}^{\ast}\right) \\
&\quad  \leq\sup_{\boldsymbol{\alpha}}\bar{r}\left(  \boldsymbol{\alpha}\right) \\
&\quad  \leq\sup_{\boldsymbol{\alpha}}\left[  \left(  2\alpha_{1}-\alpha
_{K}\right)  V_{1}^{\ast}+\sum_{\ell=2}^{K-1}\alpha_{\ell}\left(  2V_{\ell}^{\ast}-V_{\ell+1}^{\ast}\right)  +2\alpha_{K}V_{K}^{\ast}-V_{2}^{\ast}\right]
\\
&\quad  =\left(  2-\left(  1/2\right)  ^{K-1}\right)  V_{0}.
\end{align*}
The first equality follows from $2\alpha_{\ell}^{\ast}-\alpha_{\ell-1}^{\ast}=0$ for $\ell=2,\ldots,K$; the second equality from (\ref{eqn:Vaexp}); the second inequality is from (\ref{eqn:Vaexp2}); the third equality uses $\alpha_{1}=1$, $2V_{\ell}^{\ast}-V_{\ell+1}^{\ast}=0$ for $\ell=2,\ldots,K$, $-\alpha_{K}V_{1}^{\ast}+2\alpha_{K}V_{K}^{\ast}=0$ and $V_{2}^{\ast}=\left(  1/2\right)^{K-1}V_{0}$. We therefore obtain%
\[
\sup_{\boldsymbol{\alpha}}\bar{r}\left(  \boldsymbol{\alpha}\right)  =\left\{  \left(  2-\left( 1/2\right)  ^{K-1}\right)  V\left(A\right)    \right\}  .
\]
\end{proof} 

In the rest of the subsection, we will show that for any $\boldsymbol{\alpha
}\in\Delta$, both $\min_{i\in L}(R_{i}^{(1)}(\boldsymbol{\alpha})\wedge
R_{i}^{(2)}(\boldsymbol{\alpha}))$ and $\min_{i\in L}(R_{i}^{(1)}%
(\boldsymbol{\alpha})\wedge R_{i}^{(2)}(\boldsymbol{\alpha})\wedge R_{i}%
^{(3)}(\boldsymbol{\alpha}))$ are bounded below by quantities slightly smaller
than $r(\boldsymbol{\alpha})$. Actually, we will find lower bounds for
$\min_{i\in L}R_{i}^{(k)}(\boldsymbol{\alpha})$ for $k=1,2$ and $3,$
individually. The precise statement is given in the following lemma.

\begin{lemma}
\label{lem:5.5}
For any $\boldsymbol{\alpha}\in\Delta$, we have $\min_{i\in L}R_{i}%
^{(1)}(\boldsymbol{\alpha})=r(\boldsymbol{\alpha})$, $\min_{i\in L}R_{i}%
^{(2)}(\boldsymbol{\alpha})\geq r(\boldsymbol{\alpha})-h\vee w$ and
$\min_{i\in L}R_{i}^{(3)}(\boldsymbol{\alpha})\geq r(\boldsymbol{\alpha})-w$.
\end{lemma}

\begin{proof}
First note that
\begin{align*}
\min_{i\in L}R_{i}^{(1)}(\boldsymbol{\alpha})  &  =\min_{i\in L}\left(
\inf_{\boldsymbol{x}\in A\times M^{K-1}}\left\{  2f(\boldsymbol{x}%
,\boldsymbol{\alpha})+Q(O_{i},x)\right\}  +W(O_{i})-W(O_{1})\right) \\
&  =\inf_{\boldsymbol{x}\in A\times M^{K-1}}\left\{  2f(\boldsymbol{x}%
,\boldsymbol{\alpha})+\min_{i\in L}\left[  Q(O_{i},x)+W(O_{i})\right]
-W(O_{1})\right\} \\
&  =\inf_{\boldsymbol{x}\in A\times M^{K-1}}\left\{  2f(\boldsymbol{x}%
,\boldsymbol{\alpha})+W(\boldsymbol{x})-W(O_{1})\right\} \\
&  =\inf_{\boldsymbol{x}\in A\times M^{K-1}}\left\{  2f(\boldsymbol{x}%
,\boldsymbol{\alpha})+U(\boldsymbol{x})\right\}  =r(\boldsymbol{\alpha}),
\end{align*}
where we use \eqref{eqn:W_x} for the third equality and Lemma \ref{Lem:4.1}
for the fourth equality. Moreover, since  
\begin{align*}
&  \min_{i\in L\setminus\{1\}}R_{i}^{(2)}(\boldsymbol{\alpha})\\
&  \quad=\min_{i\in L\setminus\{1\}}\left[  2\inf_{\boldsymbol{x}\in A\times
M^{K-1}}[f(\boldsymbol{x},\boldsymbol{\alpha})+Q(O_{i},\boldsymbol{x}%
)]+W(O_{i})-2W(O_{1})+W(O_{1}\cup O_{i})\right] \\
&  \quad\geq \inf_{\boldsymbol{x}\in A\times M^{K-1}}[2f(\boldsymbol{x}%
,\boldsymbol{\alpha})+\min_{i\in L\setminus\{1\}}\left\{  Q(O_{i}%
,\boldsymbol{x})+W(O_{i})-W(O_{1})\right\}  ]\\
&  \qquad\qquad-W(O_{1})+\min_{i\in L\setminus\{1\}}W(O_{1}\cup O_{i})\\
&  \quad=\inf_{\boldsymbol{x}\in A\times M^{K-1}}[2f(\boldsymbol{x}%
,\boldsymbol{\alpha})+\min_{i\in L\setminus\{1\}}\left\{  Q(O_{i}%
,\boldsymbol{x})+U(O_{i})\right\}  ]-w,
\end{align*}
using $U\geq0$ we obtain
\begin{align*}
\min_{i\in L}R_{i}^{(2)}(\boldsymbol{\alpha})  &  =R_{1}^{(2)}%
(\boldsymbol{\alpha})\wedge\left(  \min_{i\in L\setminus\{1\}}R_{i}%
^{(2)}(\boldsymbol{\alpha})\right) \\
&  \geq\left(  \inf_{\boldsymbol{x}\in A\times M^{K-1}}[2f(\boldsymbol{x},\boldsymbol{\alpha})+Q(O_{1},\boldsymbol{x})]-h\right) \\
&  \qquad\wedge\left(  \inf_{\boldsymbol{x}\in A\times M^{K-1}}%
[2f(\boldsymbol{x},\boldsymbol{\alpha})+\min_{i\in L\setminus\{1\}}\left\{
Q(O_{i},\boldsymbol{x})+U(O_{i})\right\}  ]-w\right) \\
&  \geq\inf_{\boldsymbol{x}\in A\times M^{K-1}}[2f(\boldsymbol{x},\boldsymbol{\alpha})+\min_{i\in L}\left\{  Q(O_{i},\boldsymbol{x})+U(O_{i}%
)\right\}  ]-h\vee w\\
&  =\inf_{\boldsymbol{x}\in A\times M^{K-1}}[2f(\boldsymbol{x},\boldsymbol{\alpha})+U(\boldsymbol{x})]-h\vee w\\
&  =r(\boldsymbol{\alpha})-h\vee w,
\end{align*}
where the second equality is from Remark \ref{rmk:4.1}.
Lastly,
\begin{align*}
&  \min_{i\in L}R_{i}^{(3)}(\boldsymbol{\alpha})\\
&  \quad=\min_{i\in L}\left\{  2\inf_{\boldsymbol{x}\in A\times M^{K-1}%
}[f(\boldsymbol{x},\boldsymbol{\alpha})+Q(O_{i},\boldsymbol{x}%
)]+2W(O_{i})-2W(O_{1})-w\right\} \\
&  \quad=\min_{i\in L}\left\{  2\inf_{\boldsymbol{x}\in A\times M^{K-1}%
}[f(\boldsymbol{x},\boldsymbol{\alpha})+Q(O_{i},\boldsymbol{x}%
)]+2U(O_{i})\right\}  -w\\
&  \quad=2\inf_{\boldsymbol{x}\in A\times M^{K-1}}[f(\boldsymbol{x},\boldsymbol{\alpha})+\min_{i\in L}\left\{  Q(O_{i},\boldsymbol{x})+U(O_{i}%
)\right\}  ]-w\\
&  \quad=2\inf_{\boldsymbol{x}\in A\times M^{K-1}}[f(\boldsymbol{x},\boldsymbol{\alpha})+U(\boldsymbol{x})]-w\\
&  \quad\geq\inf_{\boldsymbol{x}\in A\times M^{K-1}}[2f(\boldsymbol{x},\boldsymbol{\alpha})+U(\boldsymbol{x})]-w\\
&  \quad=r(\boldsymbol{\alpha})-w.
\end{align*}

\end{proof}

\section{Bounds on the error terms $h$ and $w$}
\label{sec:6}
Lemma \ref{lem:5.5} shows that for any collection of temperature ratios $\boldsymbol{\alpha}%
\in\Delta$, $\liminf_{\varepsilon\rightarrow0}-\varepsilon\log(T^{\varepsilon
}\cdot\text{$\mathrm{Var}_{x}(\theta_{\text{INS}}^{\varepsilon,T^{\varepsilon
}})$})$ is always bounded below by $r(\boldsymbol{\alpha})-h\vee w$.  

It remains to bound $h$ and $w$ for the INS model. Let $H$ be the index
set for equilibrium points of $V$ and let $y_{i}\in M$ be the equilibrium
corresponding to index $i\in H$. Recall that we assumed $y_{1}$ is the unique
global minimum of $V$. Let $b_{1}$ be the minimum barrier height of $y_{1}$,
namely,
\begin{equation}
b_{1}\doteq\min_{j\in H\setminus\{1\}}\hat{Q}(y_{j},y_{1}),
\label{eqn:defofbs}%
\end{equation}
where $\hat{Q}$ is the quasipotential associated with the original diffusion \eqref{eqn:dym_original}, and $\hat W$ is defined analogously to $W$ but for this process.

\begin{lemma}
\label{lem:hon}$h\doteq\min_{i\in L\setminus\{1\}}Q(O_{1},O_{i}%
)=\alpha_{K}b_{1}.$
\end{lemma}

\begin{proof}
Letting $D_{1}$ be the domain of attraction of $O_{1}$, we define 
\[
Q_{D_{1}}\left(  \boldsymbol{x},\boldsymbol{y}\right)  \doteq\inf\left\{
I_{T}\left(  \phi\right)  :\phi\left(  0\right)  =\boldsymbol{x}%
,\phi\left(  T\right)  =\boldsymbol{y},\phi\left(  t\right)  \in
D_{1}\text{ for all }0\leq t\leq T,T<\infty\right\}  .
\]
Recall that $Q\left(  \boldsymbol{x},\boldsymbol{y}\right)  $ is defined by
\[
Q\left(  \boldsymbol{x},\boldsymbol{y}\right)  \doteq\inf\left\{
I_{T}\left(  \phi\right)  :\phi\left(  0\right)  =\boldsymbol{x}%
,\phi\left(  T\right)  =\boldsymbol{y},T<\infty\text{ }\right\}  .
\]
Now since $O_{1}$ is the only equilibrium point in $D_{1}$, this implies that
\[
h\doteq\min_{i\in L\setminus\{1\}}Q(O_{1},O_{i})\geq\inf_{\boldsymbol{x}%
\in\partial D_{1}}Q_{D_{1}}\left(  O_{1},\boldsymbol{x}\right)  .
\]
Moreover, we can apply \cite[Theorem 4.3, Chapter 4]{frewen2} and (\ref{eqn:U}) to find
\begin{align*}
\inf_{\boldsymbol{x}\in\partial D_{1}}Q_{D_{1}}\left(  O_{1},\boldsymbol{x}%
\right)   &  =-\lim_{\varepsilon\rightarrow0}\varepsilon\log\left(  \frac
{\nu^{\varepsilon}\left(  \partial D_{1}\right)  }{\nu^{\varepsilon}\left(
D_{1}\right)  }\right)  =\inf_{\boldsymbol{x}\in\partial D_{1}}%
U(\boldsymbol{x})-\inf_{\boldsymbol{x}\in D_{1}}U\left(  \boldsymbol{x}\right)
\\
&  =U(O_{2})-U(O_{1})=U(O_{2})=\alpha_{K}V(y_{2})=\alpha_{K}b_{1},
\end{align*}
where $O_{2}\doteq(y_{1},\ldots,y_{1},y_{2})\in\partial D_{1}$ with $y_{2}$
being an unstable equilibrium point such that $b_{1}=\hat{Q}(y_{1}%
,y_{2})=V(y_{2})$. Thus, we have $h\geq\alpha_{K}b_{1}.$ For the other
direction, we use the definitions of $Q_{D_{1}}$ and $Q$, and we apply
\cite[Theorem 4.3, Chapter 4]{frewen2} again to find
\[
h\leq Q\left(  O_{1},O_{2}\right)  \leq Q_{D_{1}}(O_{1},O_{2}%
)=U(O_{2})-U(O_{1})=\alpha_{K}b_{1}.
\]

\end{proof}

Recall that $w\doteq W(O_{1})-\min_{i\in L\setminus\{1\}}W(O_{1}\cup O_{i})$.
We provide an upper bound for $w$ in the next lemma. To state the lemma, we
need some more definitions. Let $\hat{G}(1)$ denote the collection of graphs on $\{y_i\}_{i\in H}$ that
end at $y_{1}$. Let $\hat
{G}_{\text{m}}(1)$ denote the subset of such graphs with the property that for 
every local maximum or saddle point $y$ there is a local 
local minimum $z$ such that $\hat{Q}(y, z)=0$. 
We know that $\hat{G}_{\text{m}}(1)$ is nonempty since it
contains the optimizing $\hat{g}$ in the definition of $\hat{{W}}(y_{1})$
\cite[Lemma 4.3(a), Chapter 6]{frewen2}. Given $\hat{g}\in$ $\hat{G}_{\text{m}%
}(1)$ let $H_{\hat{g}}\subset H\backslash\{1\}$ be the indices which are
starting points, i.e., $k\in H_{\hat{g}}$ means that there is no arrow in the
graph that leads to $y_{k}$. Given $k\in H_{\hat{g}}$, let $C_{\hat{g}%
}(k)$ be the cost along the path $i_{1}=k,i_{2},\ldots,i_{m}=1$ in $\hat{g}$
leading from $k$ to $1$:%
\[
C_{\hat{g}}(k)=\sum_{j=1}^{m-1}\hat{Q}\left(  y_{i_j},y_{i_{j+1}}\right)
.
\]

\begin{lemma}
\label{lem:wbound}$w\leq K\alpha_{K}\min_{\hat{g}\in\hat{G}_{\text{m}}(1)}%
\max_{k\in H_{\hat{g}}}C_{\hat{g}}(k)$.
\end{lemma}

\begin{remark}
\label{Rmk:tight}
Note that always $\min_{\hat{g}\in\hat{G}_{\text{m}}(1)}\max_{k\in H_{\hat{g}%
}}C_{\hat{g}}(k)\leq\hat{{W}}(y_{1})$, and that $\min_{\hat{g}\in\hat
{G}_{\text{m}}(1)}\max_{k\in H_{\hat{g}}}C_{\hat{g}}(k)$ can in some cases be
much smaller than $\hat{{W}}(y_{1})$. For example, this is often the case
when $H$ is large but all equilibrium points of $V$ can reach $y_{1}$ while
passing through only a few intermediate equilibrium points. 
The lemma is useful owing to the scaling in $K$ that is obtained,
but unlike the expression for $h$ is not tight.
\end{remark}

\begin{proof}
We will show that for any $i\in L\setminus\{1\}$ and any $\hat{g}\in\hat
{G}_{\text{m}}(1)$, $Q(O_{i},O_{1})\leq\alpha_{K}\max_{k\in H_{\hat{g}}%
}C_{\hat{g}}(k)$. If this is true, then from the definition of $W(O_{1}\cup
O_{i})$ we can construct a graph to use in the definition of $W(O_{1})$ that
gives $W(O_{1})\leq W(O_{1}\cup O_{i})+Q(O_{i},O_{1})$ for any $i\in
L\setminus\{1\}$. Combining these two inequalities with the definition of $w$
in (\ref{eqn:somedefs}) complete the proof.

To prove the upper bound for $Q(O_{i},O_{1})$ we fix a graph $\hat{g}\in
\hat{G}_{\text{m}}(1)$, and note that for any $y_{\ell}$ with $\ell\in
H_{\hat{g}}$, there is a unique sequence of arrows (containing no loop) that
leads from $y_{\ell}$ to $y_{1}$ with cost $C_{\hat{g}}(\ell)$.
Furthermore, we known that in this $\hat{g}$, every local maximum or saddle
point will lead to a local minimum with zero $\hat{Q}$-cost. Using these
facts, we design a route from $O_{i}$ to $O_{1}$ through points from $(\{y_i\}_{i\in H})^K$ 
in the following way.

\begin{itemize}
\item We change only one component at a time.

\item We change the component with the largest $V$-value, and replace it by
the next equilibrium point suggested by the graph $\hat{g}$. If there is more
than one component with the largest $V$-value, then we can move any one of them.

\item Then repeat the process until all the components reach $y_{1}$,
i.e., $O_{i}$ reaches $O_{1}$.
\end{itemize}

Next we analyze the $Q$-cost for each single step. For notational convenience,
suppose without lose of generality that it is the first component that takes
the largest $V$-value. Then we will move from $(x_{1},x_{2}\dots,x_{K})$ to
some $(z_{1},x_{2},\dots,x_{K})$, with $V(x_{1})\geq V(x_{\ell})$ for all
$\ell\neq1$, and $(x_{1}\rightarrow z_{1})\in\hat{g}$. We claim that
$Q((x_{1},x_{2}\dots,x_{K}),(z_{1},x_{2}\dots,x_{K}))$ is always equal to
$\alpha_{K}\hat{Q}(x_{1},z_{1}).$

We first consider the case when $x_{1}$ is a saddle point or a local maximum
of $V$. In this case then we know that $z_{1}$ must be a local minimum of $V$
such that $\hat{Q}(x_{1},z_{1})=0$, so it is easy to see that we can construct
a zero $Q$-cost trajectory 
from $(x_{1},x_{2}\dots,x_{K})$ to $(z_{1},x_{2}\dots,x_{K})$, and this gives
\[
Q((x_{1},x_{2}\dots,x_{K}),(z_{1},x_{2}\dots,x_{K}))=0=\alpha_{K}\hat{Q}%
(x_{1},z_{1}).
\]
On the other hand, if $x_{1}$ is a local minimum of $V$, then $V(z_{1})$ must
be larger than $V(x_{1})$ (which is larger than $V(x_{\ell})$ for all
$\ell\neq1$), and hence according to the definition of $U$
\begin{align*}
Q((x_{1},x_{2}\dots,x_{K}),(z_{1},x_{2}\dots,x_{K}))  &  =U(z_{1},x_{2}%
\dots,x_{K})-U(x_{1},x_{2}\dots,x_{K})\\
&  =\alpha_{K}V(z_{1})-\alpha_{K}V(x_{1})\\
&  =\alpha_{K}\hat{Q}(x_{1},z_{1}).
\end{align*}

As a result, the overall cost for each component to reach $y_{1}$ is not
larger than $\alpha_{K}\max_{k\in H_{\hat{g}}}C_{\hat{g}}(k),$ and because
there are $K$ components in total, we conclude that $Q(O_{i},O_{1})\leq
K\alpha_{K}\max_{k\in H_{\hat{g}}}C_{\hat{g}}(k)$. We then minimize on
$\hat{g}\in\hat{G}_{\text{m}}(1)$.
\end{proof}

\begin{remark}
\label{rem:defofB}
A consequence of Lemmas \ref{lem:hon} and
\ref{lem:wbound} is that if we pick the temperature ratios to be $\boldsymbol{\alpha
}^{\ast}=(1,1/2,\dots,(1/2)^{K-1})$, then $\liminf_{\varepsilon\rightarrow
0}-\varepsilon\log(T^{\varepsilon}\cdot\mathrm{Var}_{x}(\theta_{\text{INS}%
}^{\varepsilon,T^{\varepsilon}}))$ is bounded below by $2V(A)-\left(
1/2\right)  ^{K-1}(V(A)+B)$, where 
$B\doteq b_{1}\vee (K\min_{\hat{g}\in\hat{G}_{\text{m}}(1)}%
\max_{k\in H_{\hat{g}}}C_{\hat{g}}(k))$. 
For fixed $V$, the gap between this value and the best possible  $2V(A)$ decays geometrically in  $K$.
\end{remark}

\subsection{Examples}
\begin{example}
We first consider the situation depicted in Figure \ref{fig:1}. 
If we use INS with two temperatures, i.e. $K=2$ and $1=\alpha_1\geq\alpha_2>0$, then some algebra shows $h=\alpha_2 b_1=4\alpha_2$ and $w=W(O_1)-\min_{i\neq 1}W(O_1\cup O_i) = 3\alpha_2$, and therefore $h>w$. 
The outcome $h>w$ reflects the fact the well containing $y_1$ is the hardest to escape from and also contains the global minimum.
\begin{figure}[h]
    \centering
    \includegraphics[width=0.9\textwidth]{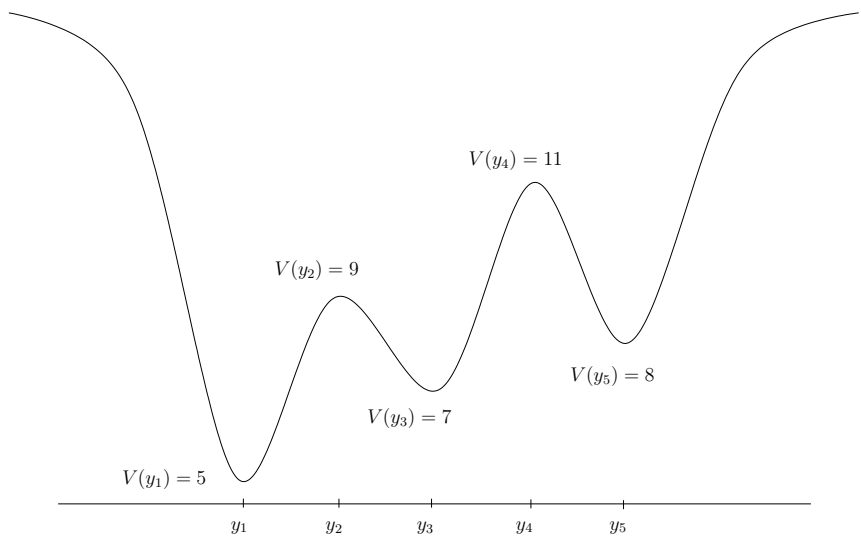}
    \caption{A case with $h>w$}
    \label{fig:1}
\end{figure}
\end{example}

\begin{example}
In this example, we consider the situation depicted in Figure \ref{fig:2}. With the same two temperature setting as in the last example, one finds $h=\alpha_2 b_1=4\alpha_2$ and $w=W(O_1)-\min_{i\neq 1}W(O_1\cup O_i) = 5\alpha_2$, which gives $w>h$. 
Here we see that there is a secondary well from which escape is harder than from that which contains $y_1$.
Moreover, in this case $\min_{\hat{g}\in\hat{G}_{\text{m}}(1)}%
\max_{k\in H_{\hat{g}}}C_{\hat{g}}(k) = \hat{W}(y_1) = 7\alpha_2$, and $K\alpha_K\min_{\hat{g}\in\hat{G}_{\text{m}}(1)}\max_{k\in H_{\hat{g}}}C_{\hat{g}}(k) = 14\alpha_2$ is strictly larger then $w=5\alpha_2$. Thus the bound for $w$ from Lemma \ref{lem:wbound} is not tight, though it is still good enough to show the deviation from optimality decays geometrically in $K$. 
\begin{figure}[h]
    \centering
    \includegraphics[width=0.9\textwidth]{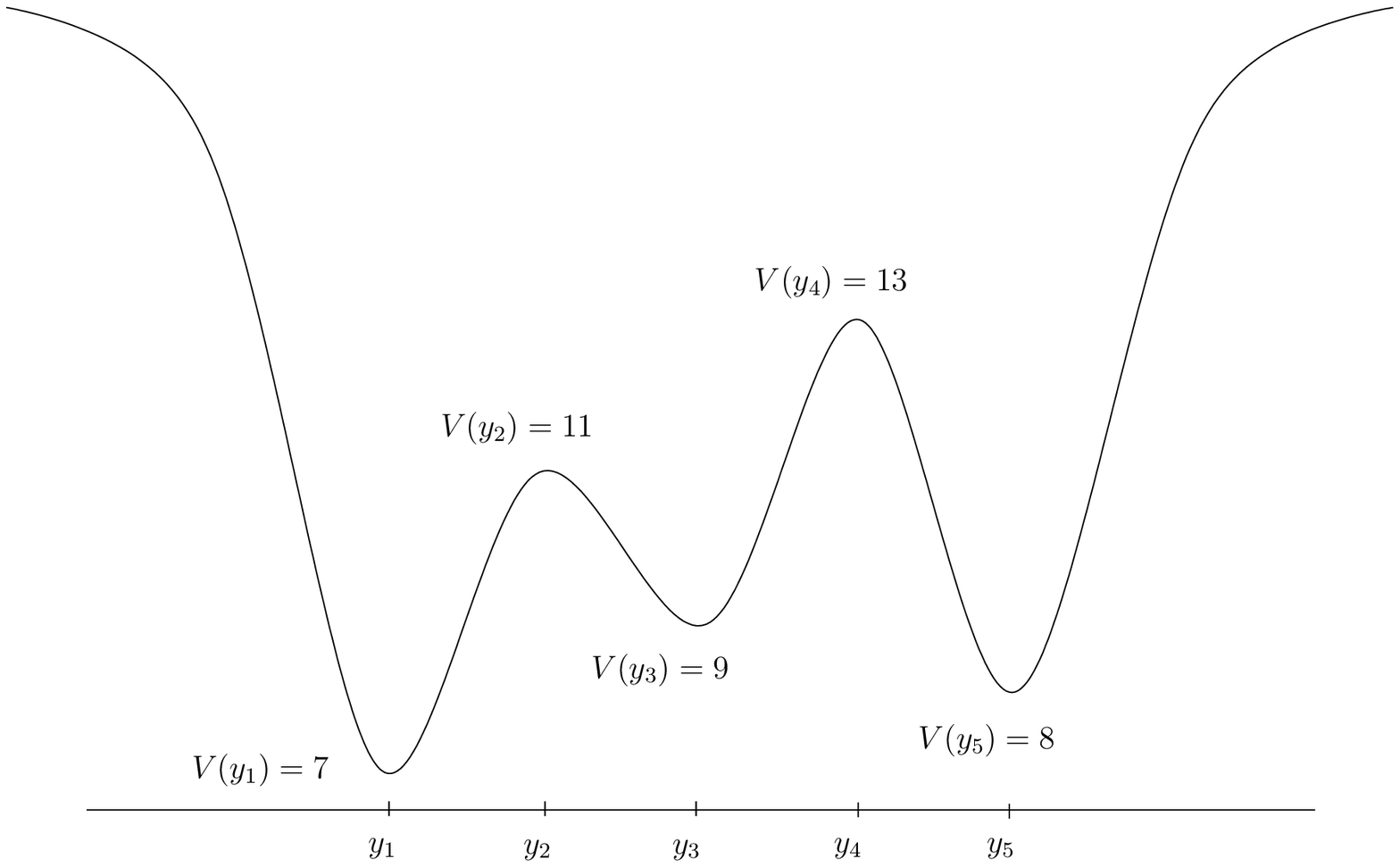}
    \caption{A case with $h<w$}
    \label{fig:2}
\end{figure}

\end{example}

\begin{example}
The last example we consider is a potential $V$ with a unique 
global minimum $y_1$ in the deepest well which is surrounded by 
$N$ collections of wells of the same form as depicted in Figure 
\ref{fig:2}, with $y_1$ common to all collections,
and each collection arranged in a radial direction out from $y_1$.
Let $\{y_i^n,i=1,\ldots,5,n=1,\ldots,N\}$ with $y^n_1=y_1$ denote the critical points of $V$.
Let $\hat g$ be the graph with all arrows pointing in along the radial direction. In this case  $H_{\hat g}$ has $N$ vertices, and with $n$ indexing such a vertex let $C_{\hat g}(n)= V(y^n_4)-V(y^n_5)+V(y^n_2)-V(y^n_3)$.
With this example,
so long as we have a uniform bound on $C_{\hat g}(n)$ there is a bound on $w$ that is independent of $N$.
Note that if there are large barriers between the radial collections then 
we will also have 
$\hat{W}(y_1) = \sum_{1\leq n \leq N}C_{\hat g}(n)$, which in this case will be much larger that $\max_{1\leq n \leq N}C_{\hat g}(n)$,
a situation noted in \ref{Rmk:tight}.
\end{example}

\section{Appendix}

The results of \cite{dupwu} use the large deviation principle for a small
noise diffusion process to characterize large deviation properties of the
variance of the empirical measure, in the limit as the time horizon tends to
infinity and the strength of the noise tends to zero. One use of the rate
function on path space is to determine probabilities of transitions between
equilibrium points of the noiseless system. As noted previously for the INS model this
is not needed, in that the known form of the stationary distribution hands us
this information directly. Because of this, all that is needed is that the LDP
holds with some rate function that is uniform with respect to initial
conditions, and certain bounds on the rate function.

One bound that is needed is an upper bound on the cost to go from any point
$\boldsymbol{x}$ to any nearby point $\boldsymbol{y}$, i.e., $\inf
\{I_{T}(\phi):\phi(0)=\boldsymbol{x},\phi(T)=\boldsymbol{y},T\in(0,\infty
)\}$, which shows that this cost can be made small by making the distance
between $\boldsymbol{x}$ and $\boldsymbol{y}$ small (a controllability type
condition). Such a bound follows easily from the non-degeneracy of the noise
and boundedness of $\nabla V$ by making comparison with the case of Brownian motion.

The other bound needed is used to show that for many calculations what happens
away from neighborhoods of the equilibrium points is not so important, in that the
process spends very little time (in a relative sense) any place but in the
union of these neighborhoods. For this, the key property of the rate function
is a result that shows that if $\delta>0$ then all zero cost trajectories
(i.e., paths $\phi$ such that $I_{T}(\phi)=0$ for all $T\in(0,\infty)$) that
start outside the union of the $\delta$-neighborhoods of the equilibrium points must
reach that set in a time that is uniformly bounded over all initial conditions
and paths.

Thus to apply the results of \cite{dupwu} two things need to be shown: an LDP
holds that is uniform with respect to initial conditions, and that if
$I_{T}(\phi)$ is the rate function for this LDP then the stability property
for zero cost paths just mentioned is true. In this section we sketch how both of these can be shown for the INS model.

\subsection{Properties of zero cost trajectories}

A condition that is sufficient to show that the time spent away from $\delta
$-neighborhoods of the equilibrium points is the following.

\begin{enumerate}
\item There is a measurable function $\bar{L}:M^{K}\times(\mathbb{R}^{d}%
)^{K}\rightarrow\lbrack0,\infty)$ that is uniformly bounded on each compact
subset, such that for all absolutely continuous $\psi\in C([0,T]:M^{K})$ $\,$,
the rate function for the INS model discussed in the next section of the
Appendix satisfies
\[
\int_{0}^{T}\bar{L}(\psi,\dot{\psi})ds\leq I_{T}(\psi),
\]
and in all other cases $I_{T}(\psi)=\infty$. 

\item For each $\delta>0$ there is $f:[0,\infty)\rightarrow\lbrack0,\infty)$
that satisfies $f(t)\rightarrow\infty$ as $t\rightarrow\infty$, and if
$\psi:[0,\infty)\rightarrow M^{K}$ is absolutely and if $\psi(t)$ avoids the
$\delta$-neighborhoods of all the equilibrium points $\{\theta_{i},i\in H\}^{K}$,
then
\begin{equation}
\int_{0}^{T}\bar{L}(\psi,\dot{\psi})ds\geq f(T). \label{eqn:RFLB}%
\end{equation}

\end{enumerate}

Given that an LDP holds with rate function $I_{T}(\phi)$, it follows from the
general large deviation upper bound proved in \cite{dupellwei} that
$I_{T}(\phi)\geq J_{T}(\phi)$, with $J_{T}(\phi)$ giving the upper bound
rate and with $J_{T}(\phi)=\int_{0}^{T}\bar{L}(\phi,\dot{\phi})ds$ of the
following form. For each point $\boldsymbol{x}\in M^{K}$ there is a finite
collection of functions
\[
H_{j}(\boldsymbol{x},\boldsymbol{\gamma})\doteq\sum_{k=1}^{K}\left[
\left\langle -\nabla V(x_{k}),\gamma_{k}\right\rangle +c_{k}^{j}\left\Vert
\gamma_{k}\right\Vert ^{2}\right]  =\sum_{k=1}^{K}\left\langle -\nabla
V(x_{k}),\gamma_{k}\right\rangle +\bar{H}_{j}(\boldsymbol{x}%
,\boldsymbol{\gamma}),
\]
$j=1,\ldots,J$, where each $\gamma_{k}\in$ $\mathbb{R}^{d}$ and for each $j$ the $c_{k}^{j}$
take distinct values from $\{\alpha_{1}^{-1},\ldots,\alpha_{K}^{-1}\}$, and
the equality defines $\bar{H}_{j}(\boldsymbol{x},\boldsymbol{\gamma})$. Note
that each $\bar{H}_{j}(\boldsymbol{x},\boldsymbol{\gamma})$ is quadratic and
positive definite (i.e., greater than zero if $\boldsymbol{\gamma}%
\neq\boldsymbol{0}$). For $\boldsymbol{\beta}=(\beta_{1},\ldots,\beta_{K})$
with each $\beta_{k}$ in the tangent space to $M$ at $x_{k}$ (the only values
where $\bar{L}(\boldsymbol{x},\boldsymbol{\beta})$ will be finite), we then
have that
\begin{align*}
\bar{L}(\boldsymbol{x},\boldsymbol{\beta})  &  =\sup_{\boldsymbol{\{\gamma
}_{k}\boldsymbol{\}}}\left[  \sum_{k=1}^{K}\left\langle \beta_{k},\gamma
_{k}\right\rangle +\sum_{k=1}^{K}\left\langle \nabla V(x_{k}),\gamma
_{k}\right\rangle -\vee_{j=1}^{J}\bar{H}_{j}(\boldsymbol{x},\boldsymbol{\gamma
})\right] \\
&  =\sup_{\boldsymbol{\{\gamma}_{k}\boldsymbol{\}}}\left[  \sum_{k=1}%
^{K}\left\langle (\beta_{k}+\nabla V(x_{k})),\gamma_{k}\right\rangle
-\vee_{j=1}^{J}\bar{H}_{j}(\boldsymbol{x},\boldsymbol{\gamma})\right]
\end{align*}
From standard theory of the Legendre-Fenchel transform, $\bar{L}%
(\boldsymbol{x},\boldsymbol{\beta})\geq0$ with equality if and only if
$\boldsymbol{\beta}+\boldsymbol{v}$ is in the set of subdifferentials of
$\vee_{j=1}^{J}\bar{H}_{j}(\boldsymbol{x},\boldsymbol{\gamma})$ in the
$\boldsymbol{\gamma}$ variable at $\boldsymbol{\gamma}=\boldsymbol{0}$, with
$\boldsymbol{v}$ being the vector of components $\nabla V(x_{k})$. Since the
subdifferentials of $\vee_{j=1}^{J}\bar{H}_{j}(\boldsymbol{x}%
,\boldsymbol{\gamma})$ at $\boldsymbol{\gamma}=\boldsymbol{0}$ is precisely
$\{\boldsymbol{0}\}$, we get that $\bar{L}(\phi,\dot{\phi})=0$ if and only if
each component of $\phi=(\phi_{1},\phi_{2},\ldots,\phi_{K})$ satisfies
$\dot{\phi}_{k}=-\nabla V(\phi_{k})$. Since we assume there are only finitely
many equilibrium points of $V$ it must be true that each component reaches the
$\delta$-neighborhood of one of the equilibrium points in finite time. The reference
\cite{dupellwei} also proves that $J_{T}(\phi)$ has compact level sets. Since
the equilibrium points of the combined system are just $\{\theta_{i},i\in H\}^{K}$,
the claimed property (\ref{eqn:RFLB}) follows from standard calculations (see,
e.g., \cite[Lemma 2.2, Chapter 4]{frewen2}).

\subsection{Uniform LDP on path space}

The second issue is more complicated. We want to argue the following:

\begin{itemize}
\item Let $X_{x}^{\varepsilon}$ denote the solution to the INS dynamics
(\ref{eqn:dym_INS}) with initial condition $x\in M^{K}$. Fix any
$T\in(0,\infty)$. Then $\{X_{x}^{\varepsilon}\}$ satisfies an LDP on
$C([0,T]:M^{K})$ with rate function $I_{T}$ that is uniform in $x\in M^{K}$
\cite[Section 1.2]{buddup4}.
\end{itemize}

Owing to the discontinuities in the diffusion coefficient as $\varepsilon
\rightarrow0$, the INS model falls into what are called processes with
\textquotedblleft discontinuous statistics\textquotedblright\ in the large
deviation literature. There are models with discontinuous statistics for which
very explicit expressions for the rate function are possible, but there are
also many examples where, although the existence of an LDP can be established,
a precise characterization is difficult. The INS model falls into the latter
category. We will describe in some detail one way to show the existence of an
LDP for the INS model. To explain the main points we consider the particular
case of an asymmetric two well model in dimension one, with $K=2$. An example
is the Franz potential with parameter $\theta$ depicted in Figure
\ref{fig:b1}:%
\[
V(x)=V(x;\theta)=\frac{3x^{4}-4(\theta-1)x^{3}-6\theta x^{2}}{2\theta
+1}+1,\quad x\in\mathbb{R}.
\]
For every $\theta\in\lbrack0,1]$, $V(\cdot;\theta)$ has a fixed local minimum
of zero at $x_{L}=-1$, another local minimum at $x_{R}=\theta$, and a fixed
barrier of height $1$ at $x=0$. Taking $\theta=1$ produces a symmetric two
well potential and $\theta=0$ gives a single well. As before, one should
imagine that the potential has been extended in a periodic fashion while
retaining this two local minimum structure. The symmetrized potential, which
identifies the stationary distribution for the INS dynamics, is plotted in
Figure \ref{fig:b2}. This potential has a global minimum with value $0$ at
$(-1,-1)$, local minima at $(-1,.85)$ and $(.85,-1)$, and a highest local
minimum at $(.85,.85)$.

In Figure \ref{fig:a2} we plot the regions in the pair of variables where the
diffusion coefficients for the symmetrized dynamics converge to a
discontinuous function. Away from these regions the $\rho_{ij}^{\varepsilon
}(\boldsymbol{x};\boldsymbol{\alpha})$ converge uniformly to a constant, with
limiting values $1$ and $0$.

\begin{figure}[h]
\centering
\includegraphics[width=0.9\textwidth]{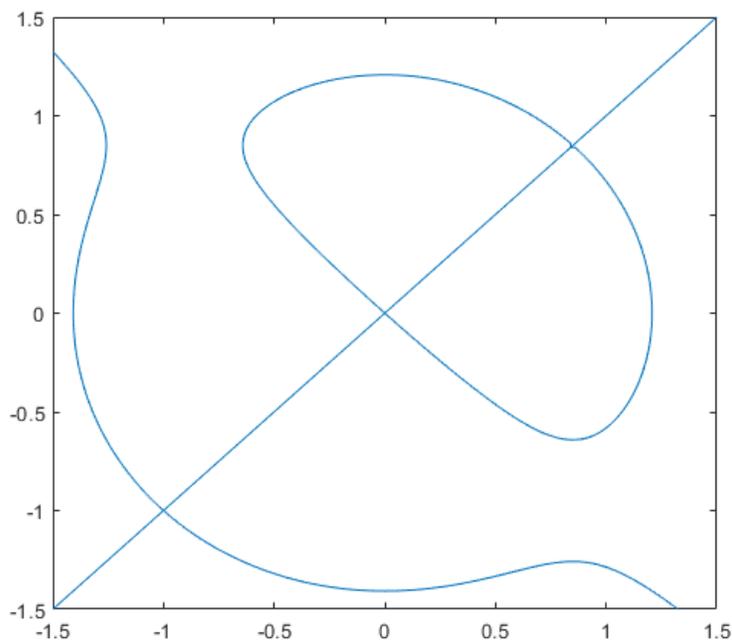} \caption{Locations where
limits of weights are discontinuous}%
\label{fig:a2}%
\end{figure}


Figure \ref{fig:b1} superimposes the locations of the critical points of the
symmetrized potential on the plat of the discontinuity region. Points
$O_{1},O_{3},O_{7}$, and $O_{9}$ are local minima with $O_{1}$ the global
minimum. Points $O_{2},O_{4},O_{6}$, and $O_{8}$ are saddle points, and
$O_{5}$ is a local maximum.

To prove the LDP one can adapt the theory presented in \cite{dupell2}, which
was motivated by problems from queueing theory and hence focuses on continuous
time processes that take values in a lattice, to deal with the diffusion
models of INS. To do so one will want some regularity assumptions on the set
$D$ of discontinuities of the functions $\lim_{\varepsilon\rightarrow
0}w^{\varepsilon}(\boldsymbol{x}_{\sigma},\boldsymbol{\alpha})$, which will
impose conditions on $V$. These discontinuities occur when two or more
$V(x_{i})$ tie, and we will want that given any point in $D$ there is a smooth
change of variable so that in an open neighborhood of the point $D$ can be
mapped to a set consisting of the union of a finite set of hyperplanes of
fixed dimension. These are mild conditions, imposing smoothness on $V$ and
ruling out sets of positive Lebesgue measure where $V$ is a constant. When
such conditions do not hold the local structure of $D$ can be more
complicated, and a more involved argument would be needed.

The method of \cite{dupell2} uses two steps to prove the LDP. One step is to
show, using the Markov property, that it is sufficient to prove large
deviation estimates of the following general form, rephrased for a continuous
state model. We suppose for simplicity of terminology that the state space is
$(\mathbb{R}^{d})^{K}$ rather than $M^{K}$.

We consider the large deviation properties of increments of the process of the
form
\[
p^{\varepsilon}(z,\Delta;\beta,\eta)\doteq P\left(  \left.  \sup_{s\in
\lbrack0,\Delta]}\left\Vert X^{\varepsilon}(s)-s\beta\right\Vert
<\eta\right\vert X^{\varepsilon}(0)=z\right)  .
\]
To establish an LDP on path space, it is sufficient to show the following. For
each $y\in(\mathbb{R}^{d})^{K}$ there is an affine space $\Gamma_{y}$ with
dimension strictly smaller than that of $(\mathbb{R}^{d})^{K}$ and a lower
semicontinuous function $L:(\mathbb{R}^{d})^{K}\times\Gamma_{y}\rightarrow
\lbrack0,\infty)$, with the property that
for each fixed $y$ the map $\beta\rightarrow L(y,\beta)$ is convex, and such
that
\begin{align}
&  \lim_{\Delta\rightarrow0}\frac{1}{\Delta}\lim_{\eta\rightarrow0}%
\lim_{\delta\rightarrow0}\liminf_{\varepsilon\rightarrow0}\inf_{\{z:\left\Vert
z-y\right\Vert \leq\delta\}}\left(  -\varepsilon\log p^{\varepsilon}%
(z,\Delta;\beta,\eta)\right) \nonumber\\
&  =\lim_{\Delta\rightarrow0}\frac{1}{\Delta}\lim_{\eta\rightarrow0}%
\lim_{\delta\rightarrow0}\limsup_{\varepsilon\rightarrow0}\sup_{\{z:\left\Vert
z-y\right\Vert \leq\delta\}}\left(  -\varepsilon\log p^{\varepsilon}%
(z,\Delta;\beta,\eta)\right) \nonumber\\
&  =L(y,\beta). \label{eqn:limits}%
\end{align}
The set $\Gamma_y$ is a local approximation to the directions in which the
dynamics of the process are in some sense uniformly (in $\varepsilon$)
continuous, and it is in directions orthogonal to $\Gamma_y$ that there are
rapidly changing or perhaps even discontinuous behaviors. We illustrate the
role of $\Gamma_y$ through the two dimensional example. The definition of
$L(y,\beta)$ for $\beta\notin\Gamma_y$ is unimportant when $\Gamma
_y\neq(\mathbb{R}^d)^K$, since the Lebesgue measure of the times $t$ where an
absolutely continuous function $\psi:[0,T]\rightarrow(\mathbb{R}^d)^K$ lies on
a hyperplane of dimension $dK-1$ and at the same time $\dot{\psi}$ is not on
that plane is zero (i.e., $\dot{\psi}(s)\in\Gamma_{\psi(s)}$ a.s.). 

Given the
estimates of (\ref{eqn:limits}) and mild regularity properties of $%
L(y,\beta)$, in the second step \cite{dupell2} shows how to combine these estimates for
increments using the Markov property to obtain a uniform LDP for
$\{X^\varepsilon\}$ on path space. (There is an error in the proof of the LDP
upper bound in \cite{dupell2} that was pointed out and corrected in \cite{ign}.)

To connect to the INS model, we consider the two temperature two well model
discussed earlier, and for which the discontinuity set $D$ is depicted in
Figure \ref{fig:a2}. There are qualitatively three types of points in this
figure: (a) continuity points, (b) points $y$ such that in a small
neighborhood of $y$ the set $D$ is smooth and one dimensional, and (c) points
$y$ such that in a small neighborhood of $y$ the set $D$ is the intersection
of two smooth, one dimensional sets. For points of type (a) we can easily show
(\ref{eqn:limits}) for $\Gamma_{y}=\mathbb{R}^{2}$ using many different
methods and with an explicit expression for $L(y,\beta)$. For points of type
(b) $\Gamma_{y}$ is the one dimensional tangent space to $D$ at $y$. Here we
do not attempt to explicitly identify $L(y,\beta)$, and the argument to
establish the existence of the limit in (\ref{eqn:limits}) uses a monotonicity
argument, a method that allows existence of limits to be shown without their
identification. For the last class of points of type (c) $\Gamma_{y}=\{0\}$.

We will describe how to prove the existence of the limits in each of the three
cases mentioned above. We recall that the INS process model is given by the solution to%
\[
\left\{
\begin{array}
[c]{l}%
dX_{1}^{\varepsilon}=-\nabla V(X_{1}^{\varepsilon})dt+\sqrt{\varepsilon}%
\sqrt{2\rho^{\varepsilon,\alpha}(X_{1}^{\varepsilon},X_{2}^{\varepsilon
})+2\rho^{\varepsilon,\alpha}(X_{2}^{\varepsilon},X_{1}^{\varepsilon})/\alpha
}dW_{1}\\
dX_{2}^{\varepsilon}=-\nabla V(X_{2}^{\varepsilon})dt+\sqrt{\varepsilon}%
\sqrt{2\rho^{\varepsilon,\alpha}(X_{1}^{\varepsilon},X_{2}^{\varepsilon
})/\alpha+2\rho^{\varepsilon,\alpha}(X_{2}^{\varepsilon},X_{1}^{\varepsilon}%
)}dW_{2}%
\end{array}
,\right.
\]
where $\alpha\in(0,1)$ and
\begin{equation}
\rho^{\varepsilon,\alpha}(x_{1},x_{2})=\frac{e^{-\frac{1}{\varepsilon}\left[
V(x_{1})+\alpha V(x_{2})\right]  }}{e^{-\frac{1}{\varepsilon}\left[
V(x_{1})+\alpha V(x_{2})\right]  }+e^{-\frac{1}{\varepsilon}\left[  \alpha
V(x_{1})+V(x_{2})\right]  }}. \label{eqn:true_rho}%
\end{equation}
Recall also that $D$ consists of points $(x_{1},x_{2})$ such that
$V(x_{1})=V(x_{2})$, and so if not in $D$ then
\[
\left[  V(x_{1})+\alpha V(x_{2})\right]  \neq\left[  \alpha V(x_{1}%
)+V(x_{2})\right]  .
\]

\subsubsection{$y\notin D$}

In this case as $\varepsilon\rightarrow0$ we have $\rho^{\varepsilon,\alpha
}(x_{1},x_{2})\rightarrow0$ or $1$ uniformly in a neighborhood of $y$. Suppose
that in fact the limit is $1$. Then by standard large deviation theory and
elementary martingale bounds the large deviation limits are the same as those
of the system
\[
\left\{
\begin{array}
[c]{l}%
dX_{1}^{\varepsilon}=-\nabla V(X_{1}^{\varepsilon})dt+\sqrt{\varepsilon}%
\sqrt{2\ }dW_{1}\\
dX_{2}^{\varepsilon}=-\nabla V(X_{2}^{\varepsilon})dt+\sqrt{\varepsilon}%
\sqrt{2/\alpha}dW_{2}%
\end{array}
,\right.
\]
i.e., (\ref{eqn:limits}) holds with
\[
L(y,\beta)=\frac{1}{4}\left[  \left(  \beta_{1}+\nabla V(y_{1})\right)
^{2}+\alpha\left(  \beta_{2}+\nabla V(y_{2})\right)  ^{2}\right]  .
\]
The analogous result holds when $\rho^{\varepsilon,\alpha}(x_{1}%
,x_{2})\rightarrow0$.

\subsubsection{$y\in D$ and locally $D$ is a smooth $1$-dimensional manifold}

We can make a smooth change of variable to \textquotedblleft
flatten\textquotedblright\ $D$ and also replace $\nabla V$ as it appears in
the drift by $(\nabla V(y_{1}),\nabla V(y_{2}))$, and $V$ as it appears in
$\rho^{\varepsilon,\alpha}(x_{1},x_{2})$ by $(V(y_{1})+(x_{1}-y_{1})\nabla
V(y_{1}),V(y_{2})+(x_{2}-y_{2})\nabla V(y_{2}))$.  The reason such
localization is relevant is because of the limit on $\Delta$ in
(\ref{eqn:limits}). This can be justified by using comparison controls to bound the differences in optimal cost under the two sets of dynamics. For notational simplicity let $b=(\nabla V(y_{1}),\nabla V(y_{2}))$.
To avoid degeneracy we will assume $b\neq0$. (If $b=0$ then the same arguments
we use below to justify the replacement of $\nabla V$ by its affine
approximation can be used to reduce to the case of $y\notin D$.)

One can check that if $b\neq0$ then $D$ is the line orthogonal to
$(-b_{1},b_{2})$. Using that $V(y_{1})=V(y_{2})$ we find
\begin{equation}
\rho^{\varepsilon,\alpha}(x_{1},x_{2})=\frac{e^{-\frac{1}{\varepsilon
}\left\langle (x-y),(b_{1},\alpha b_{2})\right\rangle }}{e^{-\frac
{1}{\varepsilon}\left\langle (x-y),(b_{1},\alpha b_{2})\right\rangle
}+e^{-\frac{1}{\varepsilon}\left\langle (x-y),(\alpha b_{1},b_{2}%
)\right\rangle }}. \label{eqn:linear_rho}%
\end{equation}
In terms of the natural coordinates defined by $(g_{1},g_{2})=\left(
-x_{1}b_{1}+x_{2}b_{2},x_{1}b_{2}+x_{2}b_{1}\right)  /\left\Vert b\right\Vert
^{2}$ we have
\begin{align*}
dG_{1}^{\varepsilon}=  &  \frac{1}{\left\Vert b\right\Vert ^{2}}\left(
b_{1}^{2}dt-b_{1}\sqrt{\varepsilon}\sqrt{2\bar{\rho}(G_{1}^{\varepsilon}/\varepsilon)+2\bar{\rho}%
(-G_{1}^{\varepsilon}/\varepsilon)/\alpha}dW_{1}\right. \\
&  \left.  \mbox{}-b_{2}^{2}dt+b_{2}\sqrt{\varepsilon}\sqrt{2\bar{\rho
}(G_{1}^{\varepsilon}/\varepsilon)/\alpha+2\bar{\rho
}(-G_{1}^{\varepsilon}/\varepsilon)}dW_{2}\right)
\end{align*}
and%
\begin{align*}
dG_{2}^{\varepsilon}=  &  \frac{1}{\left\Vert b\right\Vert ^{2}}\left(
-b_{2}b_{1}dt+b_{2}\sqrt{\varepsilon}\sqrt{2\bar{\rho}(G_{1}^{\varepsilon}/\varepsilon)+2\bar{\rho}
(-G_{1}^{\varepsilon}/\varepsilon)/\alpha}dW_{1}\right. \\
&  \left.  \mbox{}-b_{2}b_{1}dt+b_{1}\sqrt{\varepsilon}\sqrt{2\bar{\rho
}(G_{1}^{\varepsilon}/\varepsilon)/\alpha+2\bar{\rho
}(-G_{1}^{\varepsilon}/\varepsilon)}dW_{2}\right)  ,
\end{align*}
where $\bar{\rho}(g_{1})=e^{-g_{1}C}/(e^{g_{1}C}+e^{-g_{1}C})$ and $C=$
$(1-\alpha)\left[  b_{1}^{2}+b_{2}^{2}\right]  >0$.

To simplify notation we write this SDE as
\[
dG^{\varepsilon}=\bar{B}dt+\sqrt{\varepsilon}C^{\varepsilon}(G_{1}%
^{\varepsilon})dW
\]
where the diffusion matrix $C^1(g)$ is uniformly nondegenerate and  
can be written in terms of $\bar{\rho}(g/\varepsilon)$
and  $\bar{\rho}(-g/\varepsilon)$.
Note that the process depends smoothly on $G_{2}^{\varepsilon}$ (in fact owing
to the linearization it does not depend on $G_{2}^{\varepsilon}$ at all), and
the diffusion coefficient is discontinuous in $G_{1}^{\varepsilon}$ in the
limit $\varepsilon\rightarrow0$. For each $\varepsilon>0$ this SDE has a
strong solution that is unique in the strong sense. A final modification that
will ease the analysis and which is also justified by using comparison controls is to perturb $\bar{\rho}(g)$ slightly (with a controllable
change in the cost by making $\Gamma$ large), so that
\[
\bar{\rho}(g)=1\text{ for }g\geq\Gamma\text{ and }\bar{\rho}(g)=-1\text{ for }g\leq
-\Gamma.
\]

It will be enough to show that for any $y$ and $\beta=(0,\beta_{2})$ there is $L(y,\beta
)\in\lbrack0,\infty)$ such that
\begin{align}
L(y,\beta)=  &  \lim_{\eta\rightarrow0}\lim_{\delta\rightarrow0}\liminf
_{\varepsilon\rightarrow0}\inf_{\{z:\left\Vert z-y\right\Vert \leq\delta
\}}\left(  -\varepsilon\log p^{\varepsilon}(z,1;\beta,\eta)\right)
\label{eqn:limits2}\\
&  =\lim_{\eta\rightarrow0}\lim_{\delta\rightarrow0}\limsup_{\varepsilon
\rightarrow0}\sup_{\{z:\left\Vert z-y\right\Vert \leq\delta\}}\left(
-\varepsilon\log p^{\varepsilon}(z,1;\beta,\eta)\right)  .\nonumber
\end{align}
We prove (\ref{eqn:limits2}) by using stochastic control arguments and
quasistationary distributions. To be precise, we consider the controlled
system%
\[
d\bar{G}^{\varepsilon}=\bar{B}dt+C^{\varepsilon}(\bar{G}_{1}^{\varepsilon
})u^{\varepsilon}dt+\sqrt{\varepsilon}C^{\varepsilon}(\bar{G}_{1}%
^{\varepsilon})dW
\]
where $u^{\varepsilon}=(u_{1}^{\varepsilon},u_{2}^{\varepsilon})$ is any
progressively measurable, square integrable control. We make one last
elementary change, which is to absorb $\beta_{2}$ into $\bar{B}$. This can be
done since $C^{\varepsilon}$ does not depend on $g_{2}$. Let
\[
\tau^{\varepsilon}\doteq\inf\{s\geq0:\left\Vert \bar{G}^{\varepsilon
}(s)\right\Vert \geq\eta\}.
\]
Then we have the representation \cite{boudup}, \cite[Chapter 3]{buddup4}
\begin{align*}
V^{\varepsilon}(g;\eta)  &  \doteq-\varepsilon\log p^{\varepsilon}%
(y,1;\beta,\eta)\\
&  =\inf_{u^{\varepsilon}}E\left[  \left.  \frac{1}{2}\int_{0}^{1}\left\Vert
u^{\varepsilon}(s)\right\Vert ^{2}ds+\infty1_{\{\tau^{\varepsilon}\leq
1\}}\right\vert \bar{G}^{\varepsilon}(0)=g\right]  .
\end{align*}

We will need to show a type of uniform (in $\varepsilon$) continuity of
$V^{\varepsilon}(g;\eta)$ in the neighborhood $B_{\delta}(0)\doteq
\{y:\left\Vert y\right\Vert \leq\delta\}$ as described below. By a time change
and scaling properties, we can relate $V^{\varepsilon}(g;\eta)$ to a control
problem on the set $B_{1/\varepsilon}(0)$ over the time interval
$[0,1/\eta\varepsilon]$, and the dynamics%
\[
d\bar{G}=\bar{B}dt+C^{1}(\bar{G}_{1})udt+C^{1}(\bar{G}_{1})dW
\]
and the same running cost and time averaged costs, but requiring no exit
before $1/\eta\varepsilon$. If $\bar{V}^{\varepsilon}(g;1/\eta\varepsilon)$ is
the value function for this problem, then
\[
V^{\varepsilon}(g;\eta)=\bar{V}^{\varepsilon}(g/\varepsilon;1/\eta
\varepsilon),
\]
so we want a uniformity of $\bar{V}^{\varepsilon}(y;1/\eta\varepsilon)$ for
$y$ distance $\delta/\varepsilon$ from the origin.

Owing to the fact a limit $\eta\rightarrow0$, it is natural to relate $\bar
{V}^{\varepsilon}(g;1/\eta\varepsilon)$ to an ergodic control problem. For
$M\in(0,\infty)$ let $\lambda^{M}$ be the minimal cost for the ergodic control
problem when considered with these $\varepsilon=1$ dynamics and which
constrains the process to $B_{M}(0)$ with minimal cost per unit time. This
ergodic control problem is closely related to the problem of existence of a
quasistationary distribution (QSD) when the original dynamics are constrained
to $B_{M}(0)$, with the ergodic cost equal to the decay rate under the QSD,
and the QSD itself is the stationary distribution under the optimal ergodic
control. This is proved by a verification argument when a classical sense
solution to the HJB equation exists. The control problem is also related to
the existence of suitable solutions to an eigenvalue problem \cite{bernirvar}. The
required existence holds in the present setting owing to the regularity of the
boundary and smoothness and nondegeneracy of the dynamics \cite{rab}.

By the use of comparison controls it is easy to see that $\lambda^{M}$ is
nonincreasing in $M$,
\[
\lambda^{M}\downarrow\lambda^{\ast},
\]
where $\lambda^{\ast}>0$ if and only if $\bar{B}\neq0$ (note that for the ergodic control
problem we send $T\rightarrow\infty$ first). (In fact the ergodic cost is more
generally monotone in that a larger set will correspond to a smaller cost, and
hence the shape of the domain, a ball here, is not important.)  It
is easy to see that $\lambda^{\ast}$ is finite. We outline why
\[
\lim_{\eta\rightarrow0}\lim_{\delta\rightarrow0}\liminf_{\varepsilon
\rightarrow0}\inf_{\left\Vert g\right\Vert \leq\delta}\bar{V}^{\varepsilon
}(g/\varepsilon;1/\eta\varepsilon)=\lim_{\eta\rightarrow0}\lim_{\delta
\rightarrow0}\limsup_{\varepsilon\rightarrow0}\sup_{\left\Vert g\right\Vert
\leq\delta}\bar{V}^{\varepsilon}(g/\varepsilon;1/\eta\varepsilon
)=\lambda^{\ast}%
\]
is valid.

To prove the upper bound, one would argue as follows. Fix $M<\infty$. Owing to
the nondegeneracy, on an interval of the form $[0,\delta/\varepsilon\eta]$ we
can drive the process from starting points within $\delta/\varepsilon\eta$ of
zero to $B_{M/2}(0)$ with a cost of size (when averaged over the time interval
$1/\varepsilon\eta$) of size $\delta$. After this we can apply the optimal
control for the $\lambda^{M}$ problem. During the second interval of the form
$[\delta/\varepsilon\eta,1]$, ergodicity on the fixed compact set $B_{M}(0)$
gives a cost of the form $(\lambda^{M}+\delta)(1-\delta)$. One then takes
limits in the indicated order and then sends $M\rightarrow\infty$.

For the lower bound we will need to partition into cases, depending on what
happens with the $\lambda^{M}$. It is convenient here to use $R_{M}%
(0)\doteq\{(g_{1},g_{2}):\left\vert g_{1}\right\vert \vee\left\vert
g_{2}\right\vert <M\}$ rather than $B_{M}(0)$, which is possible due to
monotonicity properties mentioned previously. Let $\mu_{M}$ be the stationary
distribution under the optimal ergodic control. Suppose that for some sequence
$M_{i}\rightarrow\infty$
\[
\lim_{i\rightarrow\infty}\mu_{M_{i}}\left\{  (g_{1},g_{2}):-\Gamma
<g_{1}<\Gamma\right\}  >0.
\]
Then the optimally controlled process must return to this set repeatedly.
(When this is not the case then process will run off to $\pm\infty$ in the
$g_{1}$ direction, and this case is handled with a simpler argument.) In
this case the minimizing points of the cost potential $W^{M_{i}}(y)$ will be
uniformly bounded in $i$ (due to the need to return to $\left\{  (g_{1}%
,g_{2}):-\Gamma<g_{1}<\Gamma\right\}  $), and using comparison controls on any
fixed compact set we will have uniform bounds on the Lipschitz constant of
$W^{M_{i}}(g)$ for all large enough $i$. Hence we can pass to the limit
\[
W^{\ast}(g)=\lim_{i\rightarrow\infty}W^{M_{i}}(g).
\]
We claim that $W^{\ast}$ will satisfy the limit HJB (see \cite{agm} for properties of $\exp - W^{\ast}$ for special cases) and
\[
W^{\ast}(g)\leq\kappa\left\Vert g\right\Vert +K
\]
for some $\kappa,K<\infty$ (in fact $W^{\ast}(g)$ will be independent of
$g_{2}$). 

If the lower bound is \textit{not} true, then we know there is $a>0$ and
sequences $\eta_{j}\rightarrow0$, $\delta_{j}\rightarrow0$ with $\delta
_{j}/\eta_{j}\rightarrow0$, $g_{j}$ with $\left\Vert g_{j}\right\Vert
\leq\delta_{j}$ and $\varepsilon_{j}\rightarrow0$ such that
\begin{equation}
\bar{V}^{\varepsilon_{j}}(g_{j}/\varepsilon_{j};1/\eta_{j}\varepsilon_{j}%
)\leq\lambda^{\ast}-a \label{eqn:abound}%
\end{equation}
for all large enough $j$.

We use that $W^{\ast}(y)$ satisfies
\begin{align*}
\lambda^{\ast} &  =\left\langle DW^{\ast}(g),\bar{B}\right\rangle -\frac{1}%
{2}\left\Vert (C^{1})^T(g_{1})DW^{\ast}(g)\right\Vert ^{2}+\frac{1}{2}%
\text{tr}\left[  A(g_{1})D^{2}W^{\ast}(g)\right]  \\
&  \leq\left\langle DW^{\ast}(g),\bar{B}+C^{1}(g_{1})u\right\rangle +\frac
{1}{2}\left\Vert u\right\Vert ^{2}+\frac{1}{2}\text{tr}\left[  A(g_{1}%
)D^{2}W^{\ast}(g)\right]  ,
\end{align*}
where $A(g_{1})=C^{1}(g_{1})(C^{1})^{T}(g_{1})$. Also $\bar{V}^{\varepsilon
_{j}}(g;1/\eta_{j}\varepsilon_{j})\ $is equal to $U^{\varepsilon_{j},\eta_{j}%
}(g,t)$ at $t=0$, where $U^{\varepsilon_{j},\eta_{j}}$ satisfies
\[
\partial_{t}U^{\varepsilon_{j},\eta_{j}}(g,t)+\left\langle DU^{\varepsilon
_{j},\eta_{j}}(g,t),\bar{B}\right\rangle -\frac{1}{2}\left\Vert (C^{1})^T%
(g_{1})DU^{\varepsilon_{j},\eta_{j}}(g,t)\right\Vert ^{2}+\frac{1}{2}%
\text{tr}\left[  A(g_{1})D^{2}U^{\varepsilon_{j},\eta_{j}}(g,t)\right]  =0
\]
plus a zero terminal condition at $t=1/\varepsilon\eta$ for $g\in R_{M}(0)$
and $U^{\varepsilon_{j},\eta_{j}}(g,t)=\infty$ for $g\in\partial\left[
-1/\varepsilon,1/\varepsilon\right]  ^{2}$. (The existence and uniqueness of a
solution to this equation follows easily from the fact that $\exp
-U^{\varepsilon_{j},\eta_{j}}(g,t)$ satisfies a linear equation with zero
boundary condition.)

Now suppose that
\[
d\bar{G}=\bar{B}dt+C^{1}(\bar{G}_{1})udt+C^{1}(\bar{G}_{1})dW
\]
is an optimally controlled process for $\bar{V}^{\varepsilon_{j}}(g;1/\eta
_{j}\varepsilon_{j})$. Then
\[
dW^{\ast}(\bar{G})=\left\langle DW^{\ast}(\bar{G}),\bar{B}+C^{1}(\bar{G}%
_{1})u\right\rangle dt+\frac{1}{2}\text{tr}\left[  A(\bar{G}_{1})D^{2}W^{\ast
}(\bar{G})\right]  dt+\text{a martingale}.
\]
If the lower bound does not hold, then by (\ref{eqn:abound}) there is a
sequence of starting points $g_{j}$ such that
\[
\bar{V}^{\varepsilon_{j}}(g_{j}/\varepsilon_{j};1/\eta_{j}\varepsilon_{j}%
)\leq\lambda^{\ast}-a.
\]
That means that since $u$ is the corresponding optimal control
\[
E_{y_{j}/\varepsilon_{j},0}\left(  \varepsilon_{j}\eta_{j}\int_{0}%
^{1/\varepsilon_{j}\eta_{j}}\frac{1}{2}\left\Vert u(t)\right\Vert
^{2}dt\right)  \leq\lambda^{\ast}-a.
\]
With this control and starting point, by It\^{o}'s formula
\begin{align*}
&  E_{y_{j}/\varepsilon_{j},0}W^{\ast}(\bar{G}(1/\varepsilon_{j}\eta
_{j}))-W^{\ast}(g_{j}/\varepsilon_{j})\\
&  =E_{y_{j}/\varepsilon_{j},0}\int_{0}^{1/\varepsilon_{j}\eta_{j}%
}\left\langle DW^{\ast}(\bar{G}(t)),\bar{B}+C^{1}(\bar{G}_{1}%
(t))u(t)\right\rangle dt+\frac{1}{2}\text{tr}\left[  A(\bar{G}_{1}%
(t))D^{2}W^{\ast}(\bar{G}(t))\right]  dt\\
&  \geq \frac{1}%
{\varepsilon_{j}\eta_{j}}\ \lambda^{\ast}-E_{y_{j}/\varepsilon_{j},0}\left(  \int_{0}^{1/\varepsilon_{j}\eta_{j}%
}\frac{1}{2}\left\Vert u(t)\right\Vert ^{2}dt\right)  \geq\frac{1}%
{\varepsilon_{j}\eta_{j}}\ \lambda^{\ast}-\frac{1}{\varepsilon_{j}\eta_{j}%
}(\lambda^{\ast}-a)\geq\frac{1}{\varepsilon_{j}\eta_{j}}a.
\end{align*}
Since we have normalized so that $W^{\ast}\geq0$, $W^{\ast}(g_{j}%
/\varepsilon_{j})\geq0$. Using the upper bound $W^{\ast}(g)\leq\kappa
\left\Vert g\right\Vert +K$ and that $\bar{G}(1/\varepsilon_{j}\eta_{j})\in
R_{1/\varepsilon_{j}}(0)~$gives
\[
E_{y_{j}/\varepsilon_{j},0}W^{\ast}(\bar{G}(1/\varepsilon_{j}\eta_{j}%
))\leq\frac{\kappa}{\varepsilon_{j}}+K.
\]
Since $a>0$, since $\varepsilon_{j}\rightarrow0$ and $\eta_{j}\rightarrow0$ as
$j\rightarrow\infty$ we get a contradiction to
\[
\frac{1}{\varepsilon_{j}\eta_{j}}a\leq\frac{\kappa}{\varepsilon_{j}}+K.
\]

\subsubsection{$D$ is the intersection of two smooth $1$-dimensional
manifolds, i.e., a point}

The argument in this case is essentially the same as in the last case, except
that the only velocity we need consider is $\beta=0$, and so the centering
around this velocity is no longer needed, and the linearization is done so as to
make the $\rho^{\varepsilon,\alpha}(x_{1},x_{2})$ be of the form $e^{\frac
{1}{\varepsilon}(g_{1}B_{1}+g_{2}B_{2})}/(  e^{\frac{1}{\varepsilon
}(g_{1}B_{1}+g_{2}B_{2})}+e^{-\frac{1}{\varepsilon}(g_{1}B_{1}+g_{2}B_{2}%
)})  $.

\bibliographystyle{plain}
\bibliography{main}

\end{document}